\documentclass[12pt]{article}
\usepackage{amssymb}
\usepackage{amsmath}
\usepackage{enumerate}
\usepackage{mathrsfs}
\usepackage{graphics}
\usepackage{graphicx}
\usepackage{xcolor}
\usepackage{subfigure}
\usepackage[T1]{fontenc}
\usepackage{latexsym,amssymb,amsmath,amsfonts,amsthm}
\usepackage{txfonts}
\usepackage{ulem}
\usepackage{soul}
\usepackage{hyperref}
\hypersetup{colorlinks=true,citecolor=blue,linkcolor=black}

\newtheorem{theorem}{Theorem}[section]
\newtheorem{definition}{Definition}[section]

\newtheorem{proposition}{Proposition}[section]
\newtheorem{lemma}{Lemma}[section]

\newtheorem{remark}{Remark}[section]

\newtheorem*{thm*}{}
\newtheorem*{theorem*}{Theorem}

\begin{document}
	\renewcommand{\theequation}{\arabic{section}.\arabic{equation}}
	\begin{titlepage}
		%
		\title{\bf Non-uniqueness of Leray-Hopf Solutions to Forced Stochastic Hyperdissipative Navier-Stokes Equations up to Lions Index}
		\author{ Weiquan Chen$^{1},$\ \  Zhao Dong$^{2},$\ \ Yang Zheng$^{3}$\thanks{Corresponding author}\\
			{\small $^{1,2,3}$  Academy of Mathematics and Systems Science, Chinese Academy of Sciences, Beijing 100190, China}\\
			{\small $^{2,3}$ University of Chinese Academy of Sciences, Beijing 100049, China}\\
			({\small chenweiquan@amss.ac.cn, dzhao@amt.ac.cn, zhengyang2020@amss.ac.cn })}
		\date{}
	\end{titlepage}
	\maketitle
	\noindent\textbf{Abstract}:\
	We show non-uniqueness of local strong solutions to stochastic hyperdissipative Navier-Stokes equations with linear multiplicative noise and some certain  deterministic force.
	Such non-uniqueness holds true even if we perturb such deterministic force in appropriate sense.
	This is closely related to a critical condition on force under which Leray-Hopf solution to the stochastic equations is locally unique.
	Meanwhile, by a new idea, we show that for some stochastic force the system admits two different global Leray-Hopf solutions smooth on any compact subset of $(0,\infty) \times \mathbb{R}^d$.
%
	
	
	\noindent\textbf{Keywords}: stochastic hyperdissipative Navier-Stokes equations, non-uniqueness, Leray-Hopf solutions.

	\tableofcontents
	
	\section{Introduction and Main Results}
	We consider Cauchy problem to stochastic hyperdissipative Navier-Stokes equations on $\mathbb{R}^d $ ($d = 2, 3$):
	\begin{equation}\label{FNSE_0}
		\left\{
		\begin{aligned}
			&{\rm d} u  =\Big(- u\cdot \nabla u  + \Lambda^\alpha u - \nabla p + f \Big) {\rm d}t  + \gamma u {\rm d}{\rm B}_t,
			\\ &\nabla \cdot u = 0,
			\\ &u|_{t=0} = u_0,
		\end{aligned}
		\right.	
	\end{equation}
	where $\gamma\in\mathbb{R}$, 
	$\Lambda^\alpha:=-(-\Delta)^{\frac{\alpha}{2}}$ denotes fractional Laplacian, and ${\rm B}$ a one dimensional Brownian motion on stochastic basis $ \big(\Omega, \mathscr{F}, (\mathscr{F})_{t \geq 0}, \mathbb{P}\big) $.
	%
	The velocity field $u$ and the gradient press $p$ are to be solved given initial velocity $u_0 \in L^2_{\rm div} := \left\{ u \in L^2(\mathbb{R}^d) : \nabla \cdot u = 0 \right\}$ and $f \in L^1_t L^2_x$.
	For convenience, we call "$f$" the "force term" and "$\gamma u {\rm d}{\rm B}_t$" the "noise term".\\

	\noindent$\bullet$ {\bf 3-D Navier-Stokes equations}
	
	Existence and uniqueness of global strong solutions to 2D Navier-Stokes equations was proved by Ladyzhenskaya \cite{2DNS1969},
	while for 3-dimension the story is much longer.
	Leray \cite{leray1934} and Hopf \cite{hopf1950} constructed a class of global weak solutions (in the sense of distribution) satisfying an energy inequality, which is called Leray-Hopf weak solution in the literature.
	Meanwhile, Fujita and Kato \cite{Kato1964,kato1984strong} used heat semigroup estimates and fixed point argument to construct the so-called mild solutions.
	As a result, them also proved existence of global strong solutions for small initial datum in $ H^{\frac{1}{2}} $\cite{Kato1964} and $ L^3 $\cite{kato1984strong}.
	These spaces are invariant under natural scaling of Navier-Stokes equations and are called "critical spaces" in literature concerning local well-posedness.
	Later, mild solutions are constructed in bigger critical spaces like $\dot{B}_{p,\infty}^{-1+\frac{3}{p}}$ and $BMO^{-1}$ in \cite{chemin1999} and \cite{koch2001BMO} respectively,
	and global smoothness for small initial datum and local well-posedness are also established.
	
	Since the Leray-Hopf weak solutions may be non-unique, many studies turn to uniqueness conditions. 
	Prodi \cite{prodi1959}, Serrin \cite{serrin1962interior}, and Ladyzhenskaya \cite{ladyzenskaya1967linear} introduced an important integrability condition $L^p(0,T;L^q(\mathbb{R}^d)), \ \frac{2}{p} + \frac{d}{q} \leq 1, \ 2<p<\infty$ (known as LPS condition)
	under which Leray-Hopf solutions are unique.
	Later, \cite{masuda1984weak,kozono1996remark} considered the case $p = \infty$ and prove the uniqueness with additional right continuous condition.
	Furthermore, \cite{Meyer1996,lions2001uniqueness}
	proved the uniqueness of weak solutions if there exist a weak solution $u \in C([0,T];L^3(\mathbb{R}^3))$.
	Through decades of development, the LPS condition has evolved into the present version:
	\begin{theorem*}[\textbf{Weak-strong Uniqueness}]
		Let $\alpha =2,\ \gamma=0$.
		Let $u$ be a weak solution to (\ref{FNSE_0}) and belong to following space:
	\begin{equation}
		X_T^{p,q} := \left\{
		\begin{aligned}
			L^p(0,T;L^q)& \ \  2\leq p <\infty
			\\ C([0,T];L^q)& \ \ p=\infty
		\end{aligned}
		\right.
%
	\end{equation}
	where $2 \leq p \leq \infty, \ d\leq q \leq \infty$ satisfy $\frac{2}{p} + \frac{d}{q} \leq 1$.
	Then $u$ is the unique weak solution in this space and is a Leray-Hopf solution.
	\end{theorem*}
	
	In recent years, non-uniqueness has been studid intensivly
	in the field of Hydrodynamics PDEs since De Lellis and Sz\'ekelyhidi Jr. \cite{de2009euler,de2010admissibility,de2013dissipative,de2014dissipative} introduced the method of convex integration in
	constructing dissipative weak solutions to incompressible Euler equations. 
	A milestone in this direction is
	the proof of negative side of Onsager’s conjecture by Isett \cite{isett2018proof,buckmaster2019onsager}. For incompressible Navier-Stokes equations, convex integration was first used by Buckmaster, Colombo and Vicol \cite{buckmaster2021wild,buckmaster2020convex,buckmaster2021convex} to construct non-unique weak solutions with arbitrary pre-given energy profile. 
	Burczak, Modena and Sz\'ekelyhidi Jr. \cite{burczak2021non} showed ill-posedness for power-law flows and non-uniqueness for arbitrary $L^2_{\rm div}$-initial datum. These solutions obviously belong to the supercritical regime "$2/p + d/q > 1$". 
	As people wonder how close we could get to the critical line "$2/p+d/q = 1$", Cheskidov and Luo \cite{cheskidov20232,cheskidov2022sharp} prove
	"sharp non-uniqueness" results in space $L^p_t L^\infty$ $(1 \leq p < 2)$ for $d \geq 2$ and in $C_tL^p$ $(p < 2)$ for $d = 2$. To
	date, construction of Leray-Hopf weak solutions via convex integration method remains far from reach.
	
	%
	
	
	Very recently,
	Albritton, Bru{\'e} and Colombo \cite{Brue-NS,Brue2023gluing} construct non-unique Leray-Hopf solutions to 3D Navier-Stokes equations with specific non-zero force term. 
	In their words, the method is, intuitively, to find a "background solution" which is unstable for the Navier-Stokes dynamics in similarity variables and construct another solution via trajectories on the unstable manifold associated to the background solution.
	This approach is developed from Vishik's original work \cite{vishik2018instability-I,vishik2018instability-II} on 2D Euler equations.
	Worth mentioning,
	Albritton and Colombo \cite{AC-FNS} constructed non-unique Leray-Hopf solutions to the 2D hyperdissipative Navier-Stokes equations with specific force term.\\

	%
	%

	\noindent$\bullet$ {\bf Hyperdissipative Navier-Stokes equations}
	
	The existence of global smooth solution to hyperdissipative Navier-Stokes equations is proved by Lions \cite{lions1969quelques} for $\alpha \geq \frac{2+d}{2}$.
	The exponent $\frac{2+d}{2}$, known as Lions exponent, is suggested to be "critical" in the sense that the "total energy" $ \mathcal{E}(t) := \frac{1}{2} \|u(t)\|_{L^2}^2 + \int_{0}^{t} \|\Lambda^{\frac{\alpha}{2}}u(s)\|_{L^2}^2 {\rm d}s $ is invariant under natural scaling of hyperdissipative Navier-stokes equations whenever $\alpha = \frac{2+d}{2}$ \cite{colombo2021global} (when $d=2$, this corresponds to original 2D Navier-Stokes).
%
%
	In decades, many people spared no effort in showing the regularity of global solutions in supercritical case, i.e. $\alpha < \frac{2+d}{2}$.
	Some significant progress have been achieved in recent years. For example,
	%
	%
	Tao \cite{tao2010global} considered $\frac{ \Lambda^{(2+d)/2} }{\ln^{1/2}(2-\Delta)}$ instead of fractional Laplacian $ \Lambda^{\frac{2+d}{2}} $, and proved global regularity by a priori estimates.
	Very recently, Colombo and Haffter \cite{colombo2021global}
	made significant progress in this direction: they show that
	for any $M,\delta >0$, there exists an explicit $\varepsilon := \varepsilon(M, \delta) > 0$ such that the equations with fractional order $\alpha \in (\frac{5}{4} -\varepsilon,\frac{5}{4}]$ have a unique global smooth solution for initial datum $\|u_0 \|_{H^\delta} \leq M$.
%
	For non-uniqueness results via convex integration, we refer to \cite{li2024sharp} where sharp non-uniqueness of weak solutions in the supercritical regime($\alpha \geq 5/2$) is studied, and \cite{colombo2018ill} where non-unique Leray-Hopf weak solutions are constructed for $\alpha < 2/5$.\\
	
	\noindent$\bullet${\bf Stochastic Navier-Stokes equations}
	
	For stochastic PDEs, the compactness argument typically used in PDEs dose not apply directly.
	Viot \cite{viot1976solutions} obtained martingale weak solutions for a class of nonlinear stochastic PDEs by estimating the modulus of continuity of the approximating sequence. Later,
	Mikulevicius and Rozovskii \cite{mikulevicius2005global} obtained global martingale weak solutions for stochastic Navier-stokes equations, following\cite{Mikulevicius1999MartingalePF}.
	By a compactness and tightness argument concerning fractional Sobolev space, Flandoli and Gatarek \cite{flandoli1995martingale} proved the existence of martingale weak solutions to stochastic Navier-stokes equations with regular or coercive diffusion coefficient noise.
	Later, \cite{flandoli2008markov} constructed martingale Leray-Hopf solutions for additive noise.
	See also \cite{liu2015stochastic} for a variational frame work using Gelfand triples structure.
%
	
	An interesting topic in stochastic PDEs is that stochastic noise of suitable type may introduce regularization effects into the system, as suggested by a series of studies.
	For multiplicative noise,
	Vicol and Glatt-Holtz \cite{vicol2014local} obtained the global smooth solutions to stochastic Euler equations, and
	R\"ockner et al. \cite{noise2014} proved that the solutions to 3D Navier-stokes equations will not blow up with high probability for sufficiently small initial data.
%
	Transport noise was also found to have regularization effects. Flandoli et al. \cite{flandoli2010well} proved well-posedness of the stochastic transport equation and similar results to stochastic Navier-stokes equations were recently proven by Flandoli and Luo \cite{flandoli2021high}.
	\cite{Liu2024NonlinearNoise} also show the regularization to 3D stochastic Navier-stokes equations by nonlinear noise.
	
	
	For non-uniqueness results,
	Hofmanov\'a et al. \cite{BFH20,hofmanova2022ill,hofmanova2023nonuniqueness,hofmanova2023global} introduced convex integration into stochastic PDEs and proved non-uniqueness in law and non-uniqueness
	of Markov families for 3D stochastic Navier-Stokes equations. 
	See also \cite{chen2024sharp} where a new stochastic version of convex integration is introduced and a sharp non-uniqueness result is proved.
	%
	Non-uniqueness of Leray-Hopf solutions for stochastic Navier-Stokes equations with specific random force term is also obtained in \cite{zhu-2023non-arxiv,Zhang2023non} by developing the method 
	in \cite{Brue-NS}.\\
	%
	
	

\par Inspired by the idea of \cite{Brue-NS}, we prove that with a deterministic forcing term $f$, (\ref{FNSE_0}) admits non-unique probabilistically strong Leray-Hopf solutions. The construction in \cite{Brue-NS} relies on particular unstable background solutions (in similarity variables), any one of which determines a specific force term for the equations. In physical world, however, we have to ask the question from an opposite direction: as we are given another force term (which does not identify with that in \cite{Brue-NS}) in the equations, would the equations admit non-unique Leray-Hopf weak solutions so that in this case global regularity is falsified and singularity does exist by classical weak-strong-uniqueness? Or it is possible to give a criterion which characterizes an appropriate set of force terms that would lead to non-unique Leray-Hopf weak solutions? Also, if we are given a force term which is "close" to that in \cite{Brue-NS}, would it be able to decide that the wanted non-uniqueness holds true as supposed? The first two problems are quite intractable as they in some sense correspond to asking for a principle to divide the space of the force terms. In this paper, we address the last problem by establishing certain stability for the non-uniqueness mentioned at the beginning of this paragraph. 
	 
\par For (\ref{FNSE_0}) with stochastic force terms, we prove non-uniqueness of  global in time Leray-Hopf weak solutions by a new idea. Different from existing literature, the Leray-Hopf weak solutions are constructed so that they admit smoothness on compact subsets of $(0,+\infty)\times\mathbb R^d$.	

	\subsection{Main results}
	\begin{definition}\label{def-1.1}
		We say that the equation (\ref{FNSE_0}) has a local solution $ u $ if there exists a $\mathscr{F}_t$-stopping time $ \rho > 0 $ and a $ \mathscr{F}_t$-adapted process $ u $ 
		which satisfies the conditions:
		\\ (1) $u \in C\big([0,\rho); L^2\big)\cap L^2(0,\rho; H^{\frac{\alpha}{2} }), \  \mathbb{P}$ -$a.s.$ .
		\\ (2) $u$ satisfies (\ref{FNSE_0}) in the sense:
		\begin{equation}
			\begin{aligned}
				\langle u(t \wedge \rho) -u_0,\psi \rangle 
				=& \int_{0}^{t \wedge \rho} \langle u_i(s) u_j(s), \partial_i \psi_j \rangle {\rm d}s
				+ \int_{0}^{t \wedge \rho} \langle u(s),\Lambda^\alpha \psi \rangle {\rm d}s 
				\\ &+ \int_{0}^{t \wedge \rho} \langle f(s),\psi \rangle {\rm d}s
				+ \gamma \int_{0}^{t \wedge \rho} \langle u(s),\psi \rangle {\rm d}{\rm B}_s,
				\quad \forall\psi\in C^\infty_c (\mathbb{R}^d)~s.t.~\nabla \cdot \psi =0.
			\end{aligned}
		\end{equation}
		\\ (3) $u$ satisfies the energy inequality:
		\begin{equation}\label{Eq-energy-equ-u}
			\begin{aligned}
				&\|u(t \wedge \rho) \|_{L^2}^2 - \|u_0\|_{L^2}^2
				+ 2 \int_{0}^{t \wedge \rho} \|u(s)\|_{\dot{H}^{\frac{\alpha}{2}}}^2 {\rm d}s 
				\\ \leq& 2 \int_{0}^{t \wedge \rho} \langle f,u \rangle {\rm d}s +\gamma^2 \int_{0}^{t \wedge \rho} \|u(s)\|_{L^2}^2 {\rm d} s
				+2 \gamma \int_{0}^{t \wedge \rho} \|u(s)\|_{L^2}^2 {\rm d} {\rm B}_s \ \ \mathbb{P} - {\rm a.s.}
			\end{aligned}
		\end{equation}
	In case	$\rho = \infty$, we say that solution $u$ is global in time.
	\end{definition}
	\begin{theorem}\label{Thm-1.0}
		Let $d \in \{2,3\},\alpha \in ( 1+\frac{d}{4},1+\frac{d}{2} )$ and $ p_0 >{\frac{d}{\alpha-1} } $.
		There exist deterministic $f_0 \in L^1_t L^2_x (\mathbb{R}^+ \times \mathbb{R}^d) $
		such that for any $u_0 \in L^2_{\rm div} \cap L^{p_0}$ and $\gamma \in \mathbb{R}$, (\ref{FNSE_0}) with force term $f= f_0$ admits two local solutions $ u_1 $ and $ u_2 $ such that
		\begin{equation}
			\mathbb{P} \left( \  \|u_1(t)\|_{L^2_x} \neq \|u_2(t)\|_{L^2_x} \ {\rm for~ all} \ t \in (0,\rho) \  \right) =1
		\end{equation}
		for some positive stopping time $\rho$.
	\end{theorem}
\begin{remark}
		The scale "$1+\frac{d}{2}$" is called Lions exponent. When $ \alpha \geq 1+\frac{d}{2} $, it is not hard to show that (\ref{FNSE_0}) has unique global in time solution in the sense of definition \ref{def-1.1}. 
	\end{remark}
	\begin{remark}
		In case $d=3,\alpha=2$, i.e. 3D stochastic Navier-Stokes equations, we construct non-unique solutions in $L^r_tL^s_x$ for all $2<r<\infty,3<s<\infty$ s.t. $\frac{2}{r}+\frac{3}{s}>1$.
%
	\end{remark}
	We note that $f_0$ in Theorem \ref{Thm-1.0} satisfies $\|f_0(t)\|_{L^2} = O (t^{\frac{2+d-4\alpha}{2\alpha}})$ near $t=0$, and such $f_0$ is clearly not unique. The following result show that the non-uniqueness obtained by Theorem \ref{Thm-1.0} is stable in the sense that it still holds if we perturb $f_0$ near $t=0$ in suitable sense. 
	\begin{theorem}\label{Thm-1}
		There exists an infinite subset $A \subset L^1_t L^2_x$ such that for every $f_0 \in A$: Theorem \ref{Thm-1.0} holds true
		and there exists a constant $\varepsilon_0 = \varepsilon_0(f_0) >0$ so that the same non-uniqueness still holds for (\ref{FNSE_0}) with every force term $f\in L^1_tL^2_x$ belongs to
%
		\begin{equation}
			\left\{ f_0 +f_1 \left| \ \|f_1(t)\|_{L^2} = O(t^{\frac{2+d-4\alpha}{2\alpha} + \varepsilon_0}) \ {\rm as}\  t \ \to 0 \right. \right\}.
		\end{equation}
		Furthermore, we can find a sequence $\{f_0^{n}\}$ with $\varepsilon_0 (f_0^{n}) \to 0$.
	\end{theorem}

	\begin{remark}\label{alpha}
		If $u_0 =0 $, Theorem \ref{Thm-1} and \ref{Thm-1.0} hold for $\alpha \in ( \frac{1}{2}+\frac{d}{4},1+\frac{d}{2} )$.
	\end{remark}
		The asymptotic power $O (t^{\frac{2+d-4\alpha}{2\alpha}})$ is sharp for non-uniqueness as the following result holds:
%
	\begin{theorem}\label{Thm-2}
		Let $u_0$ be given as in Theorem \ref{Thm-1.0}. If $f \in L^1_t L^2_x$ satisfies
		$\|f(t)\|_{L^2} = O(t^{\frac{2+d-4\alpha}{2\alpha} + \varepsilon})$ as $t\rightarrow0$ for some $ \varepsilon>0$, 
		then solution to (\ref{FNSE_0}) is locally unique: there exists a $ \mathscr{F}_t$-stopping time $\rho >0$ such that for any solutions $u_1,\ u_2$, we have  $\mathbb P~\Big(u_1(\cdot\wedge\rho)=u_2(\cdot\wedge\rho)\Big)=1$. 
	\end{theorem}
	
	For (\ref{FNSE_0}) with stochastic force term, we also obtain the following non-uniqueness result. 
	\begin{theorem}\label{Thm-1.1}
		Let $d\in \{2,3\}$, $\alpha \in ( \frac{1}{2}+\frac{d}{4},1+\frac{d}{2} )$. Given $u_0=0$, $\gamma \geq 0$, there exist an $ \mathscr{F}_t$-adapted stochastic process $f\in L^1_{loc}\big(0,+\infty; L^2\big) \ \ \mathbb{P}-a.s.$ such that the equation (\ref{FNSE_0}) has two different global solutions which are smooth on compact subsets of $(0,+\infty)\times\mathbb R^d$.
	\end{theorem}	

	\subsection{Notations}
	We denote:
	\begin{equation}
		L^p_x := L^p(\mathbb{R}^d). \ \ L^2_{\rm div} := \{ f(x) \in L^2(\mathbb{R}^d) : \nabla \cdot f(x) = 0 \},\ \ H^n_{\rm div} := H^n(\mathbb{R}^d) \cap L^2_{\rm div}.
	\end{equation}
	
	Let $X,Y$ are two Banach spaces. We denote by $\mathscr{B}(X,Y)$ the bounded linear operator form $X$ to $Y$, by $\mathscr{C}(X,Y)$ the compact linear operator form $X$ to $Y$.
	Let $A \in \mathscr{B}(X,X)$. We denote by $r(A) := \sup \{ |\lambda|:\lambda \in \sigma(A) \}$ the spectral radius.
	
	Let $A$ be a densely defined linear operator on a Banach space $X$. We denote by $\sigma(A)$ the spectrum of $A$, by $\sigma_P(A)$ the point spectrum of $A$, by $\sigma_{ess}(A)$ the essential spectrum of $A$,
	by $S(A) := \displaystyle\sup \{{\rm Re}(z):z\in \sigma(A)\}$ the spectral bound of $A$.
	We call $ \sigma_d(A) := \sigma(A) / \sigma_{ess}(A)$ the discrete spectrum of $A$.
	If $\lambda \notin \sigma(A)$, we denote $ \left(\lambda {\rm Id} - A\right)^{-1} $ by $R(\lambda,A)$.
	
	Let $A$ be a generator of a strongly continuous semigroup $e^{tA}$ on $X$. We denote by $ w_0 (A) := \inf \{ a\in \mathbb{R}: \exists M< \infty,st\ \| e^{tA}\| \leq M e^{at}  \}$ the growth bound of $e^{tA}$, 
	by $w_{ess}(A) := \{ \displaystyle\inf_{t>0} \log \|e^{tA}\|_{ess} \}$ the essential growth bound where $\|e^{tA}\|_{ess} := \inf\{ \|e^{tA} -K\| : K \in \mathcal{K}(X) \}$.

	\subsection{Main steps and propositions}\label{strategy}
	
	In this subsection, we outline main steps and main propositions that leads to Theorem \ref{Thm-1.0} and \ref{Thm-1}. The strategy here is inspired by \cite{Brue-NS}. 
	
	First, we transform (\ref{FNSE_0}) into random system via the exponential formula
	%
	\begin{equation}\label{exp-tra}
		\begin{aligned}
			u = e^ {\gamma {\rm B}_t - \frac{1}{2}\gamma^2 t} v,
		\end{aligned}
	\end{equation}
	Applying It\^{o}'s formula to $v$, (\ref{FNSE_0}) is rearranged as
	\begin{equation}\label{Eq-v}
		\left\{
		\begin{aligned}
			&\partial_t {v} + h(t)^{-1} \mathbf{P} ({v} \cdot \nabla {v}) + \Lambda^\alpha {v} = h(t){\big(\underbrace{f_0+f_1}_{=f}\big)},
			\\ &\nabla \cdot {v} = 0,
			\\ &{v}|_{t=0} = {u}_0,
		\end{aligned}
		\right.	
	\end{equation}
	where $ \mathbf{P} $ denotes Leray projector and $h(t) =e^ {-\gamma {\rm B}_t +\frac{1}{2}\gamma^2 t}$. Our task is to determine a specific $f_0$ and an appropriate "range" for $f_1$. Solution to (\ref{Eq-v}) is defined in the following sense:
	\begin{definition}\label{def-1.2}
		We say that the equation (\ref{Eq-v}) has a local solution $ v $ if there exists a $\mathscr{F}_t$-stopping time $ \rho > 0 $ and a $ \mathscr{F}_t$-adapted process $ v $ 
		which satisfies the conditions:
		\\ (1) $v \in C\big([0,\rho); L^2\big)\cap L^2(0,\rho; H^{\frac{\alpha}{2}}), \  \mathbb{P}$ - {\rm a.s.}.
		\\ (2) $v$ solves (\ref{Eq-v}) in the sense:
		\begin{equation}
			\begin{aligned}
				\langle v(t \wedge \rho) -u_0,\psi \rangle 
				=& \int_{0}^{t \wedge \rho} h(s)^{-1}  \langle v_i(s) v_j(s), \partial_i \psi_j \rangle {\rm d}s
				\\ &+ \int_{0}^{t \wedge \rho} \langle v(s),\Lambda^\alpha \psi \rangle {\rm d}s 
				+ \int_{0}^{t \wedge \rho} \langle h(s)f(s),\psi \rangle {\rm d}s,\quad \forall\psi\in C^\infty_c (\mathbb{R}^d)~s.t.~\nabla\cdot\psi=0. 
			\end{aligned}
		\end{equation}
		\\ (3) $v$ satisfies the energy inequality:
		\begin{equation}\label{Eq-energy-ine-v}
			\mathbb{E} \|v(t \wedge \rho) \|_{L^2}^2 - \|u_0\|_{L^2}^2  + 2 \mathbb{E} \int_{0}^{t \wedge \rho} \|v(s)\|_{\dot{H}^{\frac{\alpha}{2}}}^2 {\rm d}s 
			\leq 2 \mathbb{E} \int_{0}^{t \wedge \rho} \langle h(s) f(s),v(s) \rangle {\rm d}s .
		\end{equation}
		We say that solution $u$ is global if $\rho = \infty$.
	\end{definition}

	Solutions to (\ref{FNSE_0}) and (\ref{Eq-v}) can be transformed into each other (see subsection \ref{subs-eq}), so we work on (\ref{Eq-v}). To adapt the strategy to Cauchy problem, we also need to subtract from (\ref{Eq-v}) the (fractional) Stokes system:
	\begin{equation}\label{eq-stokes}
		\left\{
		\begin{aligned}
			&\partial_t \mathring{u}  + \Lambda^\alpha \mathring{u}= 0,
			\\ &\nabla \cdot \mathring{u} = 0,
			\\ &\mathring{u}|_{t=0} = u_0,
		\end{aligned}
		\right.
	\end{equation}
	So by $v=\mathring{u}+\varpi$ , we turn to solve
	\begin{equation}\label{Eq-varpi}
		\left\{
		\begin{aligned}
			&\partial_t {\varpi} + h(t)^{-1} \mathbf{P}\Big((\mathring{u}+\varpi)\cdot \nabla (\mathring{u}+\varpi)\Big) + \Lambda^\alpha {\varpi} = h(t){\big(f_0+f_1\big)},
			\\ &\nabla\cdot {\varpi} = 0,
			\\ \big.&{\varpi}\big|_{t=0} = 0.
		\end{aligned}
		\right.	
	\end{equation}
	Next, we move on to the similarity coordinate given by
	\begin{equation}\label{tra-1}
		{\xi} = \frac{{x}}{ t^{\frac{1}{\alpha}} } \ , \ \tau = \text{log}\ t.
	\end{equation}
	and take substitution
	\begin{equation}\label{tra-2}
		{\mathsf W}(\xi,\tau) = t^{-\frac{1}{\alpha}+1} {\varpi}(x,t) \  ,
		\ {F_i}(\xi,\tau) = t^{2 - \frac{1}{\alpha}}{f_i}(x,t)~,~(i=0,1)
	\end{equation}
	\begin{equation}\label{tra-2'}
		\mathring{U}(\xi,\tau) = t^{-\frac{1}{\alpha}+1} {\mathring{u}}(x,t) \  ,
	\end{equation}
	then (\ref{Eq-varpi}) reads
	\begin{equation}\label{Eq-varpi'}
		\left\{
		\begin{aligned}
			&\partial_\tau {\mathsf W} -\frac{\alpha-1}{\alpha} {\mathsf W} - \frac{1}{\alpha} {\xi} \cdot \nabla_\xi {\mathsf W} + H \mathbf{P}\Big(({\mathsf W}+\mathring{U})\cdot\nabla_\xi({\mathsf W}+\mathring{U})\Big) = \Lambda^\alpha_\xi {\mathsf W} + H^{-1}{\big(\underbrace{F_0+F_1}_{=F}\big)} ,
			\\ &\nabla_\xi \cdot {\mathsf W} =0,
			\\ &\lim_{\tau \to -\infty} e^{\frac{2+d-2\alpha}{2\alpha} \tau } \|{\mathsf W}(\tau)\|_{L^2}=0,
		\end{aligned}
		\right.		
	\end{equation}
	where $\mathbf{P}={\rm Id}-\nabla\Delta^{-1}{\rm div}={\rm Id}_\xi-\nabla_\xi~\Delta^{-1}_\xi{\rm div}_\xi$ and $H = H(\tau) := h(e^\tau)^{-1} \in C(\mathbb{R})$. A subtle point here is that the limit $\displaystyle\lim_{\tau \to -\infty} e^{\frac{2+d-2\alpha}{2\alpha} \tau } \|{\mathsf W}(\tau)\|_{L^2}=0$ is equivalent to $\displaystyle\lim_{t\rightarrow0}\|\varpi(t)\|_{L^2}=0$. So this means that if we solve (\ref{Eq-varpi}) via (\ref{Eq-varpi'}), then the initial value is attached to the equations in the sense of $L^2$-limit. 
	
	\par Note that our task now is to determine $F_0$ and an appropriate range for $F_1$, for which (\ref{Eq-varpi'}) admits non-unique solutions. Construction of $F_0$ depends on a suitable divergence-free "background flow" $\bar{U}\in C^\infty_c\big(\mathbb R^d_\xi\big)$ (independent of $\tau$):
	\begin{align}\label{F_0}
		{F}_0 :=& \frac{\alpha-1}{\alpha} \bar{U} + \frac{1}{\alpha}\xi\cdot \nabla_\xi \bar{U} +\mathbf{P} (\bar{U}\cdot\nabla_\xi\bar{U}) - \Lambda_\xi^\alpha \bar{U},
	\end{align}
	and we are looking for solutions to (\ref{Eq-varpi'}) in the form
	\begin{align}\label{W=bar{U}+U_r}
		\mathsf W(\tau,\xi)=\bar{U}(\xi)+U_r(\tau,\xi).
	\end{align} 
	Inserting (\ref{W=bar{U}+U_r}) and (\ref{F_0}) into (\ref{Eq-varpi'}), we arrive at
	\begin{equation}
		\left\{
		\begin{aligned}
			&\partial_\tau U_r -\frac{\alpha-1}{\alpha} U_r - \frac{1}{\alpha} \xi \cdot \nabla_\xi U_r + H \mathbf{P}(U_r \cdot \nabla_\xi U_r + (\bar{U}+\mathring{U}) \cdot \nabla_\xi U_r +U_r \cdot \nabla_\xi (\bar{U}+\mathring{U})) 
			\\ &= \Lambda_\xi^\alpha U_r + F_r - (H-1) \mathbf{P} (\bar{U} \cdot \nabla_\xi \bar{U} ) -H \mathbf{P}(\mathring{U} \cdot \nabla_\xi \bar{U} + \bar{U} \cdot \nabla_\xi \mathring{U} ) - H \mathbf{P}(\mathring{U}\cdot\nabla_\xi\mathring{U}) ,
			\\ & \nabla_\xi \cdot U_r =0,
			\\ & \lim_{\tau \to -\infty} e^{\frac{2+d-2\alpha}{2\alpha} \tau } \|U_r(\tau)\|_{L^2}=0.
		\end{aligned}
		\right.
	\end{equation}
	with
	\begin{align}
		F_r :=& H(\tau)^{-1}F - {F}_0 = H(\tau)^{-1}F_1 +  H(\tau)^{-1}F_0 - {F}_0~.
	\end{align}
	For notational simplicity, we introduce the notations
	\begin{align}
		\mathbf{B}({U},\, {V}) &:= -\mathbf{P} (U\cdot\nabla_\xi {V})~,\nonumber\\
		\widetilde{\mathbf{B}}({U},\, {V}) &:= \mathbf{B}({U},\, {V})+ \mathbf{B}({V},\, {U})~,\nonumber\\
		\mathbf{L}(\bar{U})=\mathbf{L}_{\bar{{U}}} ( {U} ) &:= \frac{\alpha-1}{\alpha} {U} + \frac{1}{\alpha} \xi \cdot \nabla_\xi {U} -\mathbf{P} (\bar{{U}} \cdot \nabla_\xi {U} +{U} \cdot \nabla_\xi \bar{{U}}) + \Lambda_\xi^\alpha {U}~,\label{Define-L_U^-}
	\end{align}
	to write the equations in an abstract form:
	\begin{equation}\label{Eq-U_r}
		\left\{
		\begin{aligned}
			&\partial_\tau {U_r} 
			= \mathbf{L}_{\bar{{U}}}( U_r) + H \mathbf{B}( {U_r},\, {U_r} ) + \widetilde{\mathbf{B}} \big(  {U_r},\, \widetilde{{U}}  \big)
			+ \widetilde{{F}} ,
			\\ & \nabla_\xi \cdot  {U_r} =0,
			\\ & \lim_{\tau \to -\infty} e^{\frac{2+d-2\alpha}{2\alpha} \tau } \|{U_r}(\tau)\|_{L^2}=0,
		\end{aligned}
		\right.
	\end{equation}
	where $\widetilde{U} := (H-1)\bar{U}+H\mathring{U}$, $\widetilde{F} = F_r - (H-1) \mathbf{P} (\bar{U} \cdot \nabla \bar{U} ) -H \mathbf{P}(\mathring{U} \cdot \nabla \bar{U} + \bar{U} \cdot \nabla \mathring{U} ) - H \mathbf{P}(\mathring{U} \cdot \nabla\mathring{U}) $.  

	We have made a general frame work on an abstract equation, which applies to our case. The abstract equation is set as follows. Let $\mathcal H$ be a Hilbert space and $\mathcal V\hookrightarrow\mathcal H$ a Banach space. Now we are given a densely defined linear operator $\mathbf{L}$ which generates a semigroup $ e^{\tau\mathbf{L}} $ on $\mathcal H$; and a densely defined bilinear operator $\mathbf{B}:\mathcal H\times\mathcal H\rightarrow\mathcal H$	with $\mathcal D(\mathbf{B})\subset\mathcal V$ and satisfying $\|\mathbf{B}(U, V)\|_{\mathcal H}\leq C\|U\|_{\mathcal V} \|V\|_{\mathcal V}$. Also, we have some fixed $\widetilde{U}\in C_{loc}\big(\mathbb R; \mathcal V\big)$, $\widetilde{F}\in C_{loc}\big(\mathbb R; \mathcal H\big)$ and $H\in C_{loc}(\mathbb R)$. The equation reads:
	\begin{equation}\label{Eq-Normal}
		\left\{
		\begin{aligned}
			&\partial_\tau U = \mathbf{L}U + H \mathbf{B}(U,\, U) + \underbrace{\mathbf{B}(U, \widetilde{U}) + \mathbf{B}(\widetilde{U}, U)}_{=:\widetilde{\mathbf{B}}(U,\, \widetilde{U})} + \widetilde{F},\quad\tau\in\mathbb R,
			\\ & \lim_{\tau \to -\infty} e^{c\tau } \|{U}(\tau)\|_{\mathcal H}=0~for~some~c>0.
		\end{aligned}
		\right.
	\end{equation}\

	\begin{proposition}\label{prop-2.1}
		Let $\mathcal V'\hookrightarrow \mathcal V$ be a Banach space and
		$\mathbf{L},\widetilde{U},\widetilde{F}$ satisfy following conditions with some $\varepsilon,t_0>0$, $\gamma_0 \in (0,1)$ and $a_1 >a_0>0$ :\\
		\noindent {\bf (A1)} $ \displaystyle\sup_{\tau\geq 0}\left( e^{-a_0\tau}\big\|e^{\tau\mathbf{L}}\big\|_{\mathcal H \to\mathcal H}\right)< \infty$, $\displaystyle\sup_{\tau\in (0,t_0)}\left(\tau^{\gamma_0}\big\| e^{\tau\mathbf{L}}\big\|_{\mathcal H \to \mathcal V'}\right)<\infty $ ; \\
		{\bf (A2)} $\displaystyle\sup_{\tau\leq 0}\big\|e^{-\varepsilon \tau} \widetilde{U}(\tau)\big\|_{\mathcal V} \leq \infty$,
		$\displaystyle\sup_{\tau \leq 0}\big\|e^{-(a_1+\varepsilon) \tau} \widetilde{F}(\tau)\big\|_{\mathcal H} \leq \infty$ ,\\		
		\noindent then there exists $T_0 > -\infty$ such that (\ref{Eq-Normal}) has a solution $ U \in C \big( (-\infty, T_0);\mathcal V' \big) $ satisfying 
		\begin{align}
			\sup_{\tau < T_0}\big\| e^{-a_1 \tau} U(\tau)\big\|_{\mathcal V'} \leq \infty.\nonumber
		\end{align}
		\noindent Furthermore, if $\mathbf{L}$ satisfies additionally:\\
		\noindent{\bf (A3)} there exists $\eta_0 \in V'$ such that $\mathbf{L}(\eta_0) = z_0 \eta_0$ with $z_0 = a_0 +i b_0$~,
		
		\noindent then (\ref{Eq-Normal}) has two solutions $ U_1,U_2\in C \big((-\infty, T_0); V'\big)$ satisfying 
		\begin{align}
			\sup_{\tau < T_0,i=1,2}\big\| e^{-a_0 \tau} U_i(\tau)\big\|_{\mathcal V'}< \infty,\nonumber
		\end{align}
		with $T_0=\inf\bigg\{t\in\mathbb R: \displaystyle\sup_{\tau \leq t} \Big(\| \widetilde{U}(\tau) \|_{\mathcal V} + \|e^{-a_1 \tau} \widetilde{F}(\tau)\|_{\mathcal H} + |H(\tau)|\Big)\leq C_0 \bigg\}$. Here $C_0>0$ is a constant depends on spaces $\mathcal H,\mathcal V,\mathcal V'$ and operators $ \mathbf{L},\mathbf{B}$. The solutions can be represented by
		\begin{align}
			U_i(\tau)=\int_{-\infty}^{\tau} e^{(\tau -s) \mathbf{L}} \left( H(s) \mathbf{B}(U_i(s),\, U_i(s)) + \widetilde{\mathbf{B}}(U_i(s),\, \widetilde{U}(s)) + \widetilde{F}(s)\right){\rm d}s,~(i=1,2).\nonumber
		\end{align}
%
	\end{proposition}
	We are going to choose appropriate $\bar{U}\in C^\infty_c\big(\mathbb R^d_\xi\big)$ so that Theorem \ref{Thm-1.0} and \ref{Thm-1} can be derived from Proposition \ref{prop-2.1}. For this, we need another proposition given below. Proposition \ref{prop-2.1} is proved in subsection \ref{New-way} and the following Proposition \ref{prop-3.1} in next section.	
	\begin{proposition}\label{prop-3.1}
		Let $d=2,3$. Let $\mathcal H = L^2_{\rm div}$ and $\mathcal V = W^{1,p} \cap H^{\alpha'}$ with $p\in\left(2, \frac{2d}{2+d-2\alpha}\right)$ and $\alpha'\in (0,\alpha)$. Then for any $A>0$, there exists $a\in(0,A)$ and a corresponding $\bar{U}_a\in C^\infty_c\big(\mathbb R^d_\xi\big) \cap \mathcal H$ such that:
		\\{\bf (1)}
		$\mathbf{L}_{\bar{U}_a}$ has eigenvalue $z =a+ib$ for some $b \in \mathbb{R}$ ;
		\\{\bf (2)}
		$\mathbf{L}_{\bar{U}_a}$ generates a strongly continuous semigroup $e^{\tau\mathbf{L}_{\bar{U}_a}}$ on $\mathcal H$ with growth bound $ w_0 (\mathbf{L}_{\bar{U}_a}) =a$;
		\\{\bf (3)} $\displaystyle\sup_{\tau\in (0,t_0)}\left(\tau^{\gamma_0}\big\| e^{\tau\mathbf{L}_{\bar{U}_a}}\big\|_{\mathcal H \to\mathcal V}\right)< \infty$ for some $\gamma_0 \in (0,1)$ .
	\end{proposition}
	\begin{remark}\label{remark-eta-V}
		Let $\eta$ be a eigenvector of $\mathbf{L}_{\bar{U}_a}$ with eigenvalue $z'\in\mathbb{C}$,
		then it is not hard to show that $\eta \in\mathcal V$ and $e^{\mathbf{L}_{\bar{U}_a}}\eta = e^{z'} \eta$.
	\end{remark}

	\section{Proof of Prosotition \ref{prop-3.1}}\label{proof-of-prop-3.1}
	This section mainly discusses the eigenvalue problem of the operator
	$\mathbf{L}_{\bar{U}}  U := \frac{\alpha-1}{\alpha} U + \frac{1}{\alpha} \xi \cdot \nabla_\xi U -\mathbf{P} (\bar{U} \cdot \nabla U +U \cdot \nabla \bar{U}) + \Lambda^\alpha U$
	and the growth bound of the semigroup $e^{t\mathbf{L}_{\bar{U}}}$.
	
	Define operator
	\begin{equation}\label{Define-P}
		\begin{aligned}
			\mathbf{P}_\alpha U :=\frac{\alpha-1}{\alpha} U + \frac{1}{\alpha} \xi \cdot \nabla_\xi U + \Lambda^\alpha U,
		\end{aligned}		
	\end{equation}
	\begin{equation}\label{Define-L_U'}
		\mathbf{L}_{\bar{U}}' U := -\mathbf{P} (\bar{U} \cdot \nabla U +U \cdot \nabla \bar{U}).
	\end{equation}
	So we get that $\mathbf{L}_{\bar{U}} U = \mathbf{P}_\alpha U + \mathbf{P} \big(- \bar{U} \cdot \nabla U - U\cdot \nabla \bar{U}\big) = \mathbf{P}_\alpha U + \mathbf{L}_{\bar{U}}' U $.

	The main purpose of this section is to prove the Prosotition \ref{prop-3.1}.
	The proof of proposition \ref{prop-3.1} is divided to four parts:
	
	First, we need to estimate the semigroup $ e^{\mathbf{P}_\alpha t}, \ \alpha \in (\frac{1}{2}+\frac{d}{4},1+\frac{d}{2})$. $ \mathbf{P}_\alpha $ is derived from the fractional Laplacian, and naturally, we hope that it retains properties similar to those of the Laplacian. In particular, we are interested in the boundedness and regularity of the semigroup it generates. We will discuss this in subsection \ref{se2.1.1}.
	
	%
	In subsection \ref{2.1.4}, we further consider some properties of the operator $\mathbf{L}_{\bar{U}}$. In this subsection, we primarily prove Proposition \ref{prop-regularity of e^L} and \ref{prop-w_0-L_u}, which demonstrates that $ e^{\mathbf{L}_{\bar{U}}} $ exhibits regularity similar to that of $ e^{\mathbf{P}_\alpha} $, and use this to estimate $ w_{ess} (\mathbf{L}_{\bar{U}}) $. 
	By Proposition \ref{prop-w_0-L_u}, to find $\bar{U}$ satisfies conditions (1)(2) in Prpposition \ref{prop-3.1} for some $a>0$, we only need to structure $\bar{U}$ such that $\mathbf{L}_{\bar{U}}$ has a eigenvalue $z, \text{Re} (z) >0$.
	
	We will follow the approach used in \cite{Brue-Euler} and \cite{Brue-NS} to structure $\bar{U}$ such that $\mathbf{L}_{\bar{U}}$ has a eigenvalue $z, \text{Re} (z) >0$ in subsections \ref{2.1.2}. We unify the two-dimensional and three-dimensional cases and propose a generalized model which also encompasses the case of the fractional Laplacian.
	
	To prove Prosotition \ref{prop-3.1}, we still need to find $\bar{U}$ such that $w_0(\mathbf{L}_{\bar{U}}) \in (0,a)$ for all $a>0$. We will prove this in subsection \ref{2.1.3}. Actually we will prove $w_0(\mathbf{P}_\alpha+\beta\mathbf{L}_{\bar{U}})$ is continuous with respect to $\beta$ in a certain sense. We will elaborate on this in detail in Proposition \ref{prop-2.3}.
	\subsection{Properties of $e^{t \mathbf{P}_\alpha}$}\label{se2.1.1}
	In this subsection, we provide some basic properties of the operator $\mathbf{P}_\alpha$.
	
	First, we give the regularity of the hot semigroup $e^{t \Lambda^\alpha}$.
	\begin{lemma}\label{lemma-heatsemigroup}
		For any $T>0,p >2$, there exists constants $C_T>0$ such that $ \|e^{t \Lambda^\alpha} u_0\|_{W^{1,p}} \leq C_T t^{ \frac{-(2+d)p+2d}{2\alpha p} }  \|u_0\|_{L^2} $ and $ \|e^{t \Lambda^\alpha} u_0\|_{\dot{H}^{\alpha'}} \leq C_{T,\alpha'} t^{ - \frac{\alpha'}{\alpha} } \|u_0\|_{L^2} $ for any $\alpha' \in (0,\alpha)$ and $t \in (0,T)$.
		%
		%
		%
	\end{lemma}
	\begin{proof}[\textbf{proof}]
		%
		Define $f(x) := x e^{-t x^{ \frac{\alpha}{\beta}} } $ where $x \geq 0$, $\beta >0$. It can be check that $f(x)$ arrives its maximum value $t^{-\frac{\beta}{\alpha}} \big(\frac{\beta}{\alpha}\big)^{\frac{\beta}{\alpha}} e^{-\frac{\beta}{\alpha}} $ at $x = \big(\frac{\beta}{t \alpha}\big)^{\frac{\beta}{\alpha}} $.
		%
		
		Using the equation $\mathcal{F}  [e^{-t(-\Delta)^{\alpha/2}} u_0 ](\xi) = e^{ -t|\xi|^\alpha }  \mathcal{F} [u_0](\xi)$
		we can get
		\begin{equation}
			\begin{aligned}
				\|e^{-t(-\Delta)^{\alpha/2}} u_0\|_{\dot{H}^\beta} =& \| |\xi|^\beta e^{ -t|\xi|^\alpha }  \mathcal{F} [u_0] (\xi)\|_{L^2_\xi}
				\\ \leq&  \| |\xi|^\beta e^{ -t|\xi|^\alpha }\|_{L^\infty_\xi} \|\mathcal{F} [u_0] (\xi)\|_{L^2_\xi}
				\\ \leq& t^{-\frac{\beta}{\alpha}} \big(\frac{\beta}{\alpha}\big)^{\frac{\beta}{\alpha}} e^{-\frac{\beta}{\alpha}}  \|u_0\|_{L^2}.
			\end{aligned}
		\end{equation}
		In particularly $\|e^{t \Lambda^\alpha} u_0\|_{\dot{H}^{\frac{\alpha}{2}}} \lesssim t^{ - \frac{1}{2} } \|u_0\|_{L^2} $ when $\beta = \frac{\alpha}{2}$.
		%
		
		On the other hand, we know $ \|f\|_{L^p} \lesssim \|\nabla^2 f\|_{L^2}^a \|f\|_{L^2}^{1-a} $ from the Gagliardo-Nirenberg interpolation inequality for $p >2,a = \frac{(p-2)d}{4p}$.
		So we have
		\begin{equation}
			\begin{aligned}
				\|\nabla e^{-t(-\Delta)^{\alpha/2}} u_0 \|_{L^p} \lesssim& \| \nabla^3 e^{-t(-\Delta)^{\alpha/2}} u_0 \|_{L^2}^a \|\nabla e^{-t(-\Delta)^{\alpha/2}} u_0\|_{L^2}^{1-a}
				\\ \lesssim& t^{-\frac{3}{\alpha} a } t^{-\frac{1}{\alpha} (1-a) } \|u_0\|_{L^2}
				\\ \lesssim& t^{ -\frac{1}{\alpha} \frac{(2+d)p-2d}{2p} } \|u_0\|_{L^2}.
			\end{aligned}
		\end{equation}

		Similarly $\|e^{-t(-\Delta)^{\alpha/2}} u_0 \|_{L^p} \lesssim  t^{ \frac{-(p-2)d}{2\alpha p} } \|u_0\|_{L^2} \lesssim t^{ \frac{-(2+d)p+2d}{2\alpha p} } \|u_0\|_{L^2} $, so we get $\|e^{-t(-\Delta)^{\alpha/2}} u_0 \|_{W^{1,p}} \leq C_T t^{ \frac{-(2+d)p+2d}{2\alpha p} }  \|u_0\|_{L^2} $ when $t \in [0,T]$.
	\end{proof}
	
	Next we give estimates for the semigroup $e^{t\mathbf{P}_\alpha}$.
	\begin{proposition}\label{lemma-regularity-e^tP}
		$\mathbf{P}_\alpha$ 
		is generator of a semigroup $e^{\tau \mathbf{P}_\alpha}$ on $L^2_{\rm div}$ such that
		\\ (i) For all $\tau >0$:
		\begin{equation}
			\| e^{ \tau \mathbf{P}_\alpha } \|_{L^2\to L^2} \leq e^{-\frac{2+d-2\alpha}{2\alpha} \tau }.
		\end{equation}
		\\ (ii) For any $\tau_0 >0,\ \alpha' \in (0,\alpha)$, there exists $C_{\tau_0}$ such that
		\begin{equation}\label{rerularity-e^tP}
			\begin{aligned}
				&\|e^{\tau \mathbf{P}_\alpha} \|_{L^2 \to W^{1,p}}  \leq C_{\tau_0} \tau^{ \frac{-(2+d)p+2d}{2\alpha p} },
				\\ &\|e^{\tau \mathbf{P}_\alpha} \|_{L^2 \to H^{\alpha'}} \leq C_{\tau_0} \tau^{ -\frac{\alpha'}{\alpha} },
			\end{aligned}
			\ \ \tau \in (0,\tau_0).
		\end{equation}
	\end{proposition}
	\begin{proof}[\textbf{proof}]
		Consider equation:
		\begin{equation}\label{eq-2.26}
			\left\{
			\begin{aligned}
				&\partial_\tau U(\tau,\xi) = \mathbf{P}_\alpha U(\tau,\xi),
				\\ &U(\tau,\xi)|_{\tau =0 } = U_0(\xi),
			\end{aligned}
			\right.
		\end{equation}
		where $U_0(\xi) \in L^2_{\rm div} (\mathbb{R}^d)$. Then $e^{\tau \mathbf{P}_\alpha} U_0$ is the unique solution to the equation (\ref{eq-2.26}).
		Let
		\begin{equation}\label{22}
			\xi = \frac{x}{ t^{\frac{1}{\alpha}} } \ , \ \tau = \text{log}\ t,\ u(x,t) = t^{\frac{1}{\alpha}-1}  U(\xi,\tau).
		\end{equation}
		So $u(x,t)$ satisfies the following equation
		\begin{equation}
			\left\{
			\begin{aligned}
				&\partial_t u(t,x) = \Lambda^\alpha u(t,x),
				\\ &u(t,x)|_{t=1} = U_0(x).
			\end{aligned}
			\right.
		\end{equation}
		
		Notice that $\|u(x,t)\|_{L^2_x} \leq \|U_0\|_{L^2_x}$, by (\ref{22}) and Lemma \ref{lemma-heatsemigroup}:
		\begin{equation}
			\begin{aligned}
				\|U(\tau,\xi)\|_{L^2_\xi} = t^{-\frac{2+d-2\alpha}{2\alpha} } \|u(x,t)\|_{L^2_x}
				\leq t^{-\frac{2+d-2\alpha}{2\alpha} } \|U_0\|_{L^2_x}
				= e^{-\frac{2+d-2\alpha}{2\alpha} \tau } \|U_0\|_{L^2_\xi}.
			\end{aligned}
		\end{equation}
		thus
		\begin{equation}
			\| e^{ \tau P } \|_{L^2_\xi \to L^2_\xi} \leq e^{-\frac{2+d-2\alpha}{2\alpha} \tau }.
		\end{equation}
		The first part of the lemma has been proved.
		
		The norm relationship can be obtained from coordinate transformation:
		\begin{equation}\label{26}
			\|u\|_{L^p_x} = \|U\|_{L^p_\xi} t^{\frac{p+d-\alpha p}{\alpha p}}, \ 
			\|\nabla_x u\|_{L^p_x} = \|\nabla_\xi U\|_{L^p_\xi} t^{\frac{d-\alpha p}{\alpha p}}, \ 
			\|u\|_{\dot{H}^{\alpha'}_x} = \|U\|_{\dot{H}^{\alpha'}_\xi} t^{\frac{d-2 \alpha  + 2 - 2\alpha'}{2\alpha}}.
		\end{equation}
		So we can get that:
		\begin{equation}
			\begin{aligned}
				&\|\nabla_\xi U (\tau)\|_{L^p_\xi} = t^{-\frac{d-\alpha p}{\alpha p}}\|\nabla_x u\|_{L^p_x}
				\lesssim t^{-\frac{d-\alpha p}{\alpha p}} (t-1)^{ \frac{-(2+d)p+2d}{2\alpha p} } \|u_0\|_{L^2}
				\\ =& e^{-\frac{d-\alpha p}{\alpha p}\tau} (e^\tau-1)^{ \frac{-(2+d)p+2d}{2\alpha p} } \|u_0\|_{L^2}
				\lesssim \tau^{ \frac{-(2+d)p+2d}{2\alpha p} }\|u_0\|_{L^2},
			\end{aligned}
		\end{equation}
		and
		\begin{equation}
			\begin{aligned}
				\|U\|_{\dot{H}^{\alpha'}_\xi} = t^{-\frac{d-2 \alpha  + 2 - 2\alpha'}{2\alpha}} \|u\|_{\dot{H}^{\alpha'}_x} \lesssim e^{-\frac{d-2 \alpha  + 2 - 2\alpha'}{2\alpha} \tau} (e^\tau-1)^{ -\frac{\alpha'}{\alpha} }\|u_0\|_{L^2}
				\lesssim t^{-\frac{\alpha'}{\alpha}}\|u_0\|_{L^2}.
			\end{aligned}
		\end{equation}
		Thus $\|e^{\tau \mathbf{P}_\alpha} \|_{L^2 \to W^{1,p}} \leq C_{\tau_0} \tau^{ \frac{-(2+d)p+2d}{2\alpha p} } $ and $\|e^{\tau \mathbf{P}_\alpha} \|_{L^2 \to H^{\alpha'}} \leq C_{\tau_0} \tau^{ -\frac{\alpha'}{\alpha} }$ always holds for $\alpha' \in (0,\alpha)$ and $\tau \in (0,\tau_0)$.
	\end{proof}
	\begin{remark}
		In case $d=3$ and $\alpha=2$, the first part of this lemma corresponds to the Lemma 2.1 in \cite{33}. Here, we extend the applicability to arbitrary fractional Laplacian, and the proof differs from that in \cite{33} as well.
	\end{remark}
	%
	\subsection{Properties of $e^{t \mathbf{L}_{\bar{U}}}$}\label{2.1.4}
	
	In this subsection we give the regularity of $e^{t \mathbf{L}_{\bar{U}}}$ and estimate $w_{ess} (\mathbf{L}_{\bar{U}} ) $.
	First we show that $e^{t \mathbf{L}_{\bar{U}}}$ is a semigroup on $H^N_{\rm div}$ for all $N \in \mathbb{N}$ briefly.
	\begin{proposition}\label{prop-exist-semigroup}
		Let $\bar{U} \in C^\infty_c \cap L^2_{\rm div}$.
		Then $\mathbf{L}_{\bar{U}}$ is a closed, densely defined operator on $ H^N_{\rm div} $ with domain $D(\mathbf{L}_{\bar{U}}) = \{ U\in H^N_{\rm div}: \mathbf{L}_{\bar{U}} U \in H^N_{\rm div} \}$,
		$e^{t \mathbf{L}_{\bar{U}}}$ is a semigroup on $H^N_{\rm div}$ for $N \in \mathbb{N}$, and $\| e^{t \mathbf{L}_{\bar{U}}} \|_{H^N \to H^N} \leq e^{C \|\bar{U}\|_{W^{N,\infty}} t}$ where $C$ is a constant only depend on $N,d$ and $\alpha$. 
	\end{proposition}
	\begin{proof}[\textbf{proof}]
		Note that $\mathbf{L}_{\bar{U}}$ and $\mathbf{L}_{\bar{U}}^*$ are densely defined operators on $ H^N_{\rm div} $ with same domain $D(\mathbf{L}_{\bar{U}}^*) = D(\mathbf{L}_{\bar{U}})$,
		so $\mathbf{L}_{\bar{U}}$ is a closed operator.
		Since
		\begin{equation}
			\begin{aligned}
				\langle \mathbf{L}_{\bar{U}}U, U \rangle_{\dot{H}^n}
				=& \langle \mathbf{P}_\alpha U ,U \rangle_{\dot{H}^n}
				- \langle \mathbf{P} \big(\bar{U} \cdot \nabla U + U\cdot \nabla \bar{U}\big) , U \rangle_{\dot{H}^n}
				\\ \lesssim& \| U \|_{\dot{H}^n}^2
				+ \sum_{k=0}^{n-1} \left(\|D^{n-k}\bar{U}\|_{L^\infty} \|U\|_{\dot{H}^{k+1}} + \|D^{n-k+1}\bar{U}\|_{L^\infty} \|U\|_{\dot{H}^{k}}\right) \| U \|_{\dot{H}^n}
				\\ &- \langle \mathbf{P}\big(\bar{U} \cdot \nabla (\nabla^n U) \big) , (\nabla^n U) \rangle
				\\ \lesssim& \|\bar{U}\|_{W^{N,\infty}} \|U\|_{H^n}^2 + 0,
			\end{aligned}
		\end{equation}
		where $n\in \mathbb{N}$ and $U \in D(\mathbf{L}_{\bar{U}})$,
		we have $\langle \mathbf{L}_{\bar{U}}U, U \rangle_{{H}^N} \leq C\|\bar{U}\|_{W^{N,\infty}} \|U\|_{H^N}^2$ where $C$ is a constant depend on $N,d$ and $\alpha$, thus $\mathbf{L}_{\bar{U}} - C\|\bar{U}\|_{W^{N,\infty}}$ is a dissipative operator.
		By the Lumer-Phillips Theorem we know $\mathbf{L}_{\bar{U}}$ generates a semigroup on $H^N_{\rm div}$ and $\| e^{t \mathbf{L}_{\bar{U}}} \|_{H^N \to H^N} \leq e^{C \|\bar{U}\|_{W^{N,\infty}} t}$.
	\end{proof}
	Unless otherwise specified, we consider $e^{t \mathbf{L}_{\bar{U}}}$ as a semigroup on $L^2_{\rm div}$.
	
	The main results of this subsection are the following propositions:
	\begin{proposition}\label{prop-regularity of e^L}
%
		If $ p \in \left( 2, \frac{2d}{2 + d - 2\alpha} \right) $, $\alpha' \in (0,\alpha)$, then $\| e^{t\mathbf{L}_{\bar{U}}} \|_{L^2 \to W^{1,p} } \leq C_T t^{-\gamma_p}$, $\| e^{t\mathbf{L}_{\bar{U}}} \|_{L^2 \to H^{\alpha'} } \leq C_T t^{-\frac{\alpha'}{\alpha}}$ for $t \in (0,T)$, where $\gamma_p = \frac{(2+d)p - 2d}{2\alpha p} \in (0,1)$.
	\end{proposition}
	\begin{proposition}\label{prop-w_0-L_u}
		$w_{ess }(\mathbf{L}_{\bar{U}}) \leq \frac{2\alpha-2-d}{2\alpha } <0$ for any $\bar{U} \in C^\infty_c(\mathbb{R})$.
		If $w_0(\mathbf{L}_{\bar{U}}) > \frac{2\alpha-2-d}{2\alpha }$, then there exists $ z \in \sigma_d(\mathbf{L}_{\bar{U}}) $ such that $ \text{Re}(z) = w_0(\mathbf{L}_{\bar{U}}) $.
		Especially, if $\mathbf{L}_{\bar{U}}$ has a eigenvalue $\lambda, {\rm Re}(\lambda) >0$, then there exists $ z \in \sigma_d(\mathbf{L}_{\bar{U}}) $ such that $ \text{Re}(z) = w_0(\mathbf{L}_{\bar{U}}) $.
	\end{proposition}
	\begin{remark}
		We will choose $\bar{U}$ in subsection \ref{2.1.2} such that $w_0(\mathbf{L}_{\bar{U}}) > 0$.
	\end{remark}
	\begin{remark}\label{re-e^L-H^N}
		Actually by the same proof of Proposition \ref{prop-regularity of e^L} we have $\| e^{t\mathbf{L}_{\bar{U}}} \|_{H^N \to H^{N+\alpha'} } \leq C_T t^{-\frac{\alpha'}{\alpha}}$ for $t \in (0,T)$.
	\end{remark}
	
	We will prove two propositions respectively. First, we prove Proposition \ref{prop-regularity of e^L}.
	\begin{lemma}\label{lemma2.6}
		Let $f \in C([0,\infty))$, $\gamma \in (0,1)$ and $ A(t) := K t^{-\gamma}$.
		
		If for any $t > 0$:
		\begin{equation}\label{30}
			f(t) \leq A(t) + M \int_{0}^{t} A(t-s) f(s) {\rm d}s.
		\end{equation}	
		Then there exists $T>0,C_T > 0$ (not depend on $f$) such that for $t \in (0, T)$:
		\begin{equation}
			f(t) \leq C_T A(t).
		\end{equation}
	\end{lemma}
	\begin{proof}[\textbf{proof}]
		By induction, we define $A_0(t) := A(t)$ and $A_{n+1}(t) := M \int_{0}^{t} A(t-s) A_n(s) \, {\rm d}s$. A straightforward calculation shows that
		\begin{equation}
			A_n(t) = C_n t^{-\gamma + n (1-\gamma)},
		\end{equation}
		where $C_n$ is a constant that depends only on $\gamma$, $n$, and $K$.
		
		Since $\gamma < 1$, there exists $N = \lceil \frac{\gamma}{1 - \gamma} \rceil \in \mathbb{N}$ such that $A_N(t)$ is bounded by $C_N$ on $[0, 1]$, where $\lceil \cdot \rceil$ denotes the ceiling function.
		Note that $A_n(t)$ is continuous and integrable on $(0, 1)$ with a singularity for any $n < N$.

		Let $g_f(t) := f(t) - \sum_{i=0}^{N-1} A_i(t)$. Clearly $ \displaystyle\max_{s \in [0, t]} g_f(s) <\infty$ for all $t>0$ since $f(t)$ is continuous.
		We now show that $\displaystyle\max_{s \in [0, T]} g_f(s)$ is bounded by some constant independent of $f$ when $T$ is small enough.
		By (\ref{30}) we have:
		\begin{equation}
			\begin{aligned}
				g_f(t) \leq A_N( t) + M \int_{0}^{t} A(t-s) g_f(s) {\rm d}s 
				\leq C_N+ \frac{KM}{1-\gamma}t^{1-\gamma}\sup_{s \in [0,t]} g_f(s).
			\end{aligned}			
		\end{equation}		
		Let $ T $ be sufficiently small such that $ T^{1-\gamma} \leq \frac{1-\gamma}{2KM} $, then
		\begin{equation}
			\displaystyle\max_{s \in [0, T]} g_f(s) \leq \frac{C_N}{ 1- \frac{KM}{1-\gamma} T^{1-\gamma}}=:C_T'.
		\end{equation}
%
%
		Together by $A_n(t) \leq \frac{C_n}{K}T^{n(1+\gamma)} A(t) $, we have $f(t) \leq \big(\sum_{n=1}^{N-1} \frac{C_N}{K} T^{n(1+\gamma)} + \frac{C_T'}{K}T^{\gamma}\big)A(t) $.
%
	\end{proof}
	

	\begin{proof}[\textbf{Proof of Proposition \ref{prop-regularity of e^L}}]
		Recall $\mathbf{L}_{\bar{U}}= \mathbf{P}_\alpha+  \mathbf{L}_{\bar{U}}' $, where $\mathbf{P}_\alpha$ and $\mathbf{L}_{\bar{U}}'$ are defined in (\ref{Define-P}) and (\ref{Define-L_U'}).
%
%
		Let the space $\mathcal{V}$ be $W^{1,p}, $$\gamma := \frac{(2+d)p-2d}{2\alpha p}$, $ p \in \left( 2, \frac{2d}{2 + d - 2\alpha} \right) $ or $\mathcal{V}$ be $H^{\alpha'},$$\gamma := \frac{\alpha'}{\alpha}$, $ \alpha' \in (0,\alpha)$. By Proposition \ref{lemma-regularity-e^tP} we have $\| e^{t \mathbf{P}_\alpha} \|_{L^2 \to \mathcal{V}} \leq C_T t^{-\gamma}$ for $t \in [0,T]$.
		
		Let $T(t)u_0$ be the solution of the following equation:
		\begin{equation}
			\left\{
			\begin{aligned}
				&\partial_t u = \mathbf{P}_\alpha u + \mathbf{L}_{\bar{U}}' u,
				\\ &u|_{t=0} = u_0.
			\end{aligned}
			\right.
		\end{equation}
		The well-posedness is guaranteed by $\bar{U} \in C^\infty_c$ and Proposition \ref{prop-exist-semigroup}.
		Notice that
		\begin{equation}
			\begin{aligned}
				T(t)u_0 = e^{t\mathbf{P}_\alpha}u_0 + \int_{0}^{t} e^{(t-s) \mathbf{P}_\alpha} \mathbf{L}_{\bar{U}}' T(s) u_0 {\rm d}s,
			\end{aligned}
		\end{equation}
		and $\|\mathbf{L}_{\bar{U}}'\|_{V \to L^2} \leq \|\bar{U}\|_{L^2} + \|\bar{U}\|_{W^{1,\infty}} =:  M$.
		Therefore,
		\begin{equation}
			\begin{aligned}
				\|T(t) u_0\|_{\mathcal{V}} 
				\leq& \| e^{t\mathbf{P}_\alpha} \|_{L^2\to \mathcal{V}} \|u_0\|_{L^2} + M \int_{0}^{t} \| e^{(t-s) \mathbf{P}_\alpha} \|_{L^2 \to \mathcal{V}}\|T(s)u_0\|_{\mathcal{V}} {\rm d}s.
			\end{aligned}
		\end{equation}
		
		First, consider the case when $ u_0 \in \mathcal{V} $. 
		Obviously $ f(t) := \|T(t)u_0\|_{\mathcal{V}} $ is continuous on $[0,\infty)$. 
		We may assume without loss of generality that $ \|u_0\|_{L^2} = 1 $, by Proposition \ref{lemma-regularity-e^tP} we have:
		\begin{equation}
			\begin{aligned}
				f(t) \leq& \| e^{t\mathbf{P}_\alpha} \|_{L^2\to \mathcal{V}} 
				+ M \int_{0}^{t} \| e^{(t-s) \mathbf{P}_\alpha} \|_{L^2\to \mathcal{V}} f(s) {\rm d}s 
				\\ \leq& Kt^{-\gamma} + M\int_{0}^{t} (t-s)^{-\gamma} f(s) {\rm d}s .
			\end{aligned}
		\end{equation}
		By applying Lemma \ref{lemma2.6} there exists $T,C_T > 0$ such that
		\begin{equation}\label{2.40}
			\|T(t)u_0\|_{\mathcal{V}} \leq C_T t^{-\gamma} \|u_0\|_{L^2}, \quad t \in [0,T].
		\end{equation}
		By a density argument, (\ref{2.40}) holds for all $ u_0 \in L^2_{\rm div} $.
%
%
%
	\end{proof}
	
	To prove Proposition \ref{prop-w_0-L_u}, we need the following lemma.
	\begin{lemma}\label{lemma2.3}
		Let $\mathbf{A}$ be the generator of a semigroup $e^{t\mathbf{A}}$. If $ w_{ess}(\mathbf{A}) < \text{Re}(\lambda) $ for some $ \lambda \in \sigma(\mathbf{A}) $, then $ \exists a \in \sigma_d(\mathbf{A}) $ such that $ w_0(\mathbf{A})=\text{Re}(a) $.
	\end{lemma}
	\begin{proof}[\textbf{proof}]
		%
		From \cite{One-P}, we know that $ w_0(\mathbf{A}) = \displaystyle\max\{ S(\mathbf{A}), w_{ess}(\mathbf{A}) \} $. Since $ w_{ess}(\mathbf{A}) < \text{Re}(a) \leq S(\mathbf{A}) $, it follows that $ w_0(\mathbf{A}) \neq w_{ess}(\mathbf{A}) $.
		Then existence of $a_A$ follows from the proof of Theorem 2.11 in \cite{One-P}.
		%
	\end{proof}
%
	\begin{proof}[\textbf{Proof of Proposition \ref{prop-w_0-L_u}}]
		For notational convenience, we denote ${\mathcal H} := L^2_{\rm div}, {\mathcal V}:= H^1_{\rm div}$, $K := \inf \{ r: {\rm supp}(\bar{U}) \in B_r \}$, and 
		\begin{equation*}
			\begin{aligned}
				&\mathcal{H}_K := \{ f \in {\mathcal H} \mid \text{supp}(f) \subset B_K \},
				\\ &\mathcal{H}_{R} := \{ f \in {\mathcal H} \mid \text{supp}(f) \subset B_R \}.
			\end{aligned}
		\end{equation*}
		%
		Since $e^{s \mathbf{L}_{\bar{U}}} \in \mathscr{B} (\mathcal{H};\mathcal{V})$ and $\mathbf{L}_{\bar{U}}' \in \mathscr{B} (\mathcal{V};\mathcal{H}_K)$ (note that ${\rm supp} (\mathbf{L}_{\bar{U}}'u) \subset B_K, \ \forall u \in \mathcal{V} $), we see $\mathbf{L}_{\bar{U}}' e^{s \mathbf{L}_{\bar{U}}} \in \mathscr{B} (\mathcal{H};\mathcal{H}_K)$.
		

		Denote by $ g $ the heat kernel of the fractional heat equation, i.e. $\widehat{g}{(t,\xi)} = e^{-|\xi|^\alpha t}$, $ e^{t \Lambda^\alpha} u_0(x) = \int_{\mathbb{R}^d} g(t,x -y) u_0(y) {\rm d} y $. Then $ \| g(t) \|_{L^\infty(B_R^c)} \to 0$ as $ R \to \infty $.
		Let $u_0 \in H_K $, then
		\begin{equation}
			\begin{aligned}
				\| e^{t \Lambda^\alpha} u_0 \|_{L^2(B_R^c)}^2 =& \left\| g(t) * u_0 \right\|_{L^2(B_R^c)}^2
				\\ =& \int_{B_R^c} \left(\int_{\mathbb{R}^d} g(t,x -y)1_{\{|x-y|>R-K\}} u_0(y) {\rm d} y \right)^2 {\rm d} x
				\\ \leq& \| g(t)\|_{L^\infty(B_{R-K}^c)}^2 \|u_0\|_{L^2}^2
				\stackrel{R \to \infty}{\longrightarrow}  0.
			\end{aligned}
		\end{equation}
		%
		This indicates $\| e^{t \mathbf{P}_\alpha} \|_{L^2 \to L^2(B_R^c)} \stackrel{R \to \infty}{\longrightarrow}  0$ for all $t>0$.
		Moreover, by Lemma \ref{lemma-heatsemigroup}, $ e^{t \mathbf{P}_\alpha} $ is a compact operator from $ \mathcal{H}_K $ to $ \mathcal{H}_R $ for any $ R $.
		
		Now we take a bounded sequence $\{u_n\} $ in $ \mathcal{H}_K$ with $\|u_n\|_{L^2} \leq C$.
%
		Then by
%
		\begin{equation}
			\begin{aligned}
				&\| e^{t \mathbf{P}_\alpha}  (u_n -u_m) \|_{L^2}
				\leq \| e^{t \mathbf{P}_\alpha}  (u_n -u_m) \|_{L^2(B_R)}
				+ \| e^{t \mathbf{P}_\alpha}  (u_n -u_m) \|_{L^2(B_R^c)}
				\\ \leq& \| e^{t \mathbf{P}_\alpha}  (u_n -u_m) \|_{L^2(B_R)}
				+ 2C \| e^{t \mathbf{P}_\alpha}\|_{L^2 \to L^2(B_R^c)},
%
			\end{aligned}
		\end{equation}
		we have
		\begin{equation}
			\lim_{n,m \to \infty} \| e^{t \mathbf{P}_\alpha}  (u_n -u_m) \|_{L^2} \leq 2C \| e^{t \mathbf{P}_\alpha}\|_{L^2 \to L^2(B_R^c)}, \ \forall R >0.
		\end{equation}
		By sending $R \to \infty$, we see that $\{e^{t \mathbf{P}_\alpha} u_n\}$ is Cauchy in $\mathcal{H}$.
		Consequently, $e^{(t-s) \mathbf{P}_\alpha}\in \mathscr{C}(\mathcal{H}_K , \mathcal{H})$, $\mathbf{L}_{\bar{U}}' e^{s \mathbf{L}_{\bar{U}}} \in \mathscr{B}(\mathcal{H} , \mathcal{H}_K)$
		and so that $e^{(t-s) \mathbf{P}_\alpha} \mathbf{L}_{\bar{U}}' e^{s \mathbf{L}_{\bar{U}}} \in \mathscr{C}(\mathcal{H}, \mathcal{H})$.
		According to \cite{One-P} Thm-C.7, it follows that $\int_{0}^{t} e^{(t-s) \mathbf{P}_\alpha} \mathbf{L}_{\bar{U}}' e^{s \mathbf{L}_{\bar{U}}} \, {\rm d}s \in \mathscr{C}(\mathcal{H},\mathcal{H}) $.
		%
		To this end, we remind that
		\begin{equation}
			\begin{aligned}
				e^{t \mathbf{L}_{\bar{U}}} = e^{t \mathbf{P}_\alpha} + \int_{0}^{t} e^{(t-s) \mathbf{P}_\alpha } \mathbf{L}_{\bar{U}}' e^{s \mathbf{L}_{\bar{U}}} {\rm d}s.
			\end{aligned}
		\end{equation}
		Then by the definition of $ \|\cdot \|_{ess}$ and Proposition \ref{lemma-regularity-e^tP}, it follows that $\| e^{t \mathbf{L}_{\bar{U}}}  \|_{ess} \leq \| e^{t \mathbf{P}_\alpha}\| \leq C e^{t \frac{2\alpha-2-d}{2\alpha }} $ for all $t>0$.
		%
		By Proposition 2.10 in \cite{One-P}, we can estimate:
		\begin{equation}
			w_{ess }(\mathbf{L}_{\bar{U}})= \lim_{t\to \infty} \frac{1}{t} \log \|e^{t \mathbf{L}_{\bar{U}}} \|_{ess} \leq \frac{2\alpha-2-d}{2\alpha } <0 .
		\end{equation}
		Now if $ \mathbf{L}_{\bar{U}} $ has an eigenvalue $ \lambda $ with positive real part, then 
		$ w_{\text{ess}}(\mathbf{L}_{\bar{U}}) < 0 \leq \rm{Re}(z) \leq S(\mathbf{L}_{\bar{U}}).$
		Applying Lemma \ref{lemma2.3}, it follows  immediately that there exists $ z \in \sigma_P (\mathbf{L}_{\bar{U}}) $ with $ \text{Re}(z) = w_0(\mathbf{L}_{\bar{U}}) $.		
	\end{proof}
	
	%
	%
	%
	%
	%

	\subsection{The existence of eigenvalues}	\label{2.1.2}
	We first prove Prosotition \ref{prop-3.1} for some $a>0$ in this subsection:
	
	\begin{proposition}\label{prop-3.1-a}
		Let $d=2,3$. Let $\mathcal H = L^2_{\rm div}$ and $\mathcal V = W^{1,p} \cap H^{\alpha'}$ with $p\in\left(2, \frac{2d}{2+d-2\alpha}\right)$ and $\alpha'\in (0,\alpha)$. Then there exists $a > 0$ and a corresponding $\bar{U}_a\in C^\infty_c\big(\mathbb R^d_\xi\big)$ such that:
		\\{\bf (1)}
		$\mathbf{L}_{\bar{U}_a}$ has eigenvalue $z =a+ib$ for some $b \in \mathbb{R}$ ;
		\\{\bf (2)}
		$\mathbf{L}_{\bar{U}_a}$ generates a strongly continuous semigroup $e^{\tau\mathbf{L}_{\bar{U}_a}}$ on $\mathcal H$ with growth bound $ w_0 (\mathbf{L}_{\bar{U}_a}) =a$;
		\\{\bf (3)} $\displaystyle\sup_{\tau\in (0,t_0)}\left(\tau^{\gamma_0}\big\| e^{\tau\mathbf{L}_{\bar{U}_a}}\big\|_{\mathcal H \to\mathcal V}\right)< \infty$ for some $\gamma_0 \in (0,1)$ .
	\end{proposition}
	
	\noindent From Proposition \ref{prop-regularity of e^L} and Proposition \ref{prop-w_0-L_u}, it remains to construct $\bar{U}$ such that $\mathbf{L}_{\bar{U}} = \mathbf{P}_\alpha + \mathbf{L}_{\bar{U}}'$ has eigenvalues with positive real parts.  
	Based on \cite{Brue-Euler, Brue-NS}, we already know this is true for $\mathbf{L}_{\bar{U}}'$.  
	We should show that $\mathbf{P}_\alpha$, seen as a perturbation to operator $ \mathbf{L}_{\bar{U}}'$, makes no essential changes.
	Specifically, we need the following proposition to prove Proposition \ref{prop-3.1-a}.
	The following is a generalized version of Theorem 2.4.2 in \cite{Brue-Euler}, the proof of the argument is coming from Theorem 2.4.2 in \cite{Brue-Euler}.
	\begin{proposition}\label{prop2.2}
%
		Let $\mathcal{H}$ be a Hilbert space and $V$ a dense subset of $\mathcal{H}$. Let $\mathbf{K} \in \mathscr{C}(\mathcal{H}), \mathbf{S}$ and $ \mathbf{P} $ be generators of some semigroups on $\mathcal{H}$.
		If following conditions are satisfied:
		\\
		{\bf (1)} $\mathbf{L} := \mathbf{K}+ \mathbf{S}$ has eigenvalue with real part $a_0 >0$ and $w_0(\mathbf{L}) =a_0 $,
		\\
		{\bf (2)} $\langle  \mathbf{S} U ,U \rangle 
		\leq b \|U \|^2$ for all $U \in V$, with some $b < a_0$,
		\\
		{\bf (3)} $\displaystyle\sup_{t \in [0,T]} \|\mathbf{P} e^{t \mathbf{S}}U \|_{\mathcal{H}} < \infty$ for all $U \in V$ and $T < \infty$,
		\\
		then there exists $\beta_0 \in (0,1)$ such that $\mathbf{L}_{\beta_0} := \mathbf{L} +\beta_0 \mathbf{P}$ has a eigenvalue with real part $z_0' >0$.
%
	\end{proposition}
	The proof of Proposition \ref{prop2.2} is shown in Appendix \ref{appendix-1}.
	
	\begin{proof}[\textbf{Proof of Proposition \ref{prop-3.1-a}}]
%
%
%
%
%
		By writing
		\begin{equation}
			\begin{aligned}
				\mathbf{L}_{\bar{U}} U = \underbrace{- \mathbf{P} \big(\bar{U}\cdot \nabla U + U \cdot \nabla \bar{U} \big)}_{ =: \mathbf{L}_{\bar{U}}' U}+ \mathbf{P}_\alpha U
				= \frac{1}{\beta} \left( \underbrace{- \mathbf{P} \big(\beta\bar{U}\cdot \nabla U + U \cdot \nabla \beta\bar{U}\big) }_{= \mathbf{L}_{\beta\bar{U}}' U} +  \beta\mathbf{P}_\alpha U \right) ,
			\end{aligned}
		\end{equation}
		we see that if 
		\begin{equation*}
			\textbf{(H)}\ \text{there exists}\  g \in C^\infty_c \text{and}\ \beta >0\ \text{s.t.} \
			\mathbf{L}_g' + \beta \mathbf{P}_\alpha \ \text{has eigenvalue with positive real part,}
		\end{equation*}
		then with $\bar{U} =\beta^{-1} g$, the operator $\mathbf{L}_{\bar{U}}$ has eigenvalue with positive real part.
		
		However, Proposition \ref{prop2.2} does not apply directly, since it is not immediately to substract a compact operator from $\mathbf{L}_{g}'$ so that the remaining part satisfies the required conditions.
		In spirit of \cite{Brue-NS}, we consider similar transform, which preserves eigenvalue: 
		\begin{equation}
			\begin{aligned}
				{\rm Curl} \circ \mathbf{L}_{\beta^{-1}g} \circ {\rm Curl}^{-1} 
				= \frac{1}{\beta} \left( \underbrace{ {\rm {\rm Curl}}_d \circ \mathbf{L}_g' \circ {\rm Curl}^{-1} }_{=: \mathbb{L}_g'}
				+ \beta \underbrace{ {\rm Curl} \circ \mathbf{P}_\alpha \circ {\rm Curl}^{-1} }_{= \mathbf{P}_\alpha} \right),
			\end{aligned}
		\end{equation}
		where ${\rm Curl}_d$ is the curl operator on $\mathbb{R}^d$.
		We will use Proposition \ref{prop2.2} to get the eigenvalue of $\mathbb{L}_\beta :=\mathbb{L}_g'+ \beta\mathbf{P}_\alpha$.
		We decompose as:
		\begin{equation}
			\begin{aligned}
				\mathbb{L}_g' U=& {\rm {\rm Curl}}_d \circ \mathbf{L}_\beta^{-1}g' \circ {\rm Curl}^{-1} U
				\\ =& - {\rm {\rm Curl}}_d \left(g \cdot \nabla ({\rm Curl}^{-1}U) + ({\rm Curl}^{-1}U) \cdot \nabla g \right)
				\\ =& \underbrace{- ({\rm Curl}^{-1}U)\cdot \nabla {\rm Curl}g}_{=:\mathbb{K} U}
				\underbrace{ - \varepsilon_{i,j} \left(  \partial_i g \cdot \nabla ({\rm Curl}^{-1}U)_j - \partial_i ({\rm Curl}^{-1}U) \cdot \nabla g_j  \right) }_{=:\mathbb{S} U}
				\underbrace{- g \cdot \nabla U}_{=:\mathbb{M} U}
				 ,
			\end{aligned}
		\end{equation}
%
%
		where $\mathbb{M}$ is a skew-adjoin operator and $\mathbb{S}$ is a bounded operator on $L^2(\mathbb{R}^d)$. We will choose invariant subspace $H \subset L^2_{\rm div}$ of $\mathbb{K}, \ \mathbb{S} + \mathbb{M}$ and $\mathbf{P}_\alpha$, such that $\mathbb{K}$ is a compact operator on $H$. 
		
		We note that it is not hard to show 
		\begin{equation}
			\langle( \mathbb{S+M}) U,U\rangle_{H^N} \leq C(\|g\|_{H^{N+1}} + \|D^{N+1}g\|_{L^\infty}) \|U\|_{H^N},
		\end{equation}
		\begin{equation}
			\sup_{t\in [0,T]}\|\mathbf{P}_\alpha e^{t(\mathbb{S+M})} \eta\|_{L^2} \leq C e^{TC(\|g\|_{H^{4}} + \|D^{4}g\|_{L^\infty}) }  \|\eta\|_{H^3} <\infty ,
		\end{equation}
		and Condition (3) is satisfied.

		For $d=2$: We note that $\mathbb{S} =0 $, $\langle\mathbb{M} U,U\rangle =0$, and hence Condition (2) hold trivially; we choose $g$ as in Proposition 2.4.2 of \cite{Brue-Euler} and $H := L^2_m = \{f \in : L^2_{\rm div} : R_{\frac{2\pi}{m}} f=f\}$ where $R_{\theta}$ is counterclockwise rotation of angle $\theta$ around the origin, then $\mathbb{K}$ is a compact operator on $H$ and Condition (1) holds.
		
		For $d=3$: We choose $g$ as in Proposition 2.6 of \cite{Brue-NS}, so that $\mathbb{L}_g'$ has a eigenvalue $z_0 = a_0 + i b_0$ with $w_0(\mathbb{L}_g') =a_0 >0$ (i.e. Condition (1)). 
		Let $H:= L^2_{aps}$ denote the set of $L_2-$integrable "axisymmetric pure-swirl" vector fields $U = U^\theta(r,z)e^\theta$ in \cite{Brue-NS}, then $\mathbb{K}$ is a compact operator on $H$.
		As shown by Lemma 3.1 and Theorem 3.2 in \cite{Brue-NS}, we have $\|\mathbb{S}\|_{L^2_{aps} \to L^2} < S(\mathbb{L}_g') \leq w_0(\mathbb{L}_g') = a_0$.
		Condition (2) is satisfied.

		Using Proposition \ref{prop2.2}, it immediately follows that there exists $ \beta \in (0,1)$ such that $\mathbb{L}_\beta $ has a eigenvalue $ z_\beta = a_\beta + i b_\beta,\ a_\beta>0 $ with corresponding eigenfunction $ \eta_\beta $.
		Note that $\mathbf{P}_\alpha U =  U + \frac{1}{\alpha} \xi \cdot \nabla_\xi U + \Lambda^\alpha U$, so $e^{t \mathbf{P}_\alpha} \in \mathscr{B}(L^2_{\rm div} , H^N_{\rm div})$ and $e^{t \mathbb{L}_\beta} \in \mathscr{B}(L^2_{\rm div} , H^N_{\rm div})$ for all $t>0, \ N \in \mathbb{N}$. By $e^{t \mathbb{L}_\beta}  \eta_\beta = e^{z_\beta t} \eta_\beta $ we have $\eta_\beta \in H^N$, so ${\rm Curl}^{-1} \eta_\beta \in L^2_{\rm div}(\mathbb{R}^d)$.
		
		Since $ \mathbf{L}_{\beta^{-1}g} {\rm Curl}^{-1} \eta_\beta = \beta {\rm Curl}^{-1} (\mathbb{L}_\beta \eta_\beta) = \beta {\rm Curl}^{-1} a_\beta \eta_\beta $, it follows that $ {\rm Curl}^{-1} \eta_\beta $ is an eigenfunction of $ \mathbf{L}_{\beta^{-1}g} $ with positive real eigenvalue.
		By Proposition \ref{prop-regularity of e^L} we finish the proof of Proposition \ref{prop-3.1-a}.
%
%
	\end{proof}

	\subsection{Proof of Proposition \ref{prop-3.1}} \label{2.1.3}
	We fix the $g(\xi)$ as in Proposition \ref{prop-3.1-a} so that $\mathbf{L}_{g}$ has eigenvalue with positive real part.
	To prove Proposition \ref{prop-3.1}, we only need to prove the following proposition:
	\begin{proposition}\label{prop-2.3}
		Let $\mathcal{W}_0^+ := \{\beta \in \mathbb{R}:w_0(\mathbf{L}_{\beta g}) \geq 0\}$ where $ \mathbf{L}_{\beta g} = \mathbf{P}_\alpha + \beta \mathbf{L}_{g}' ,\beta\in \mathbb{R} $.
		We have that
		$\displaystyle\lim_{\beta \to \beta_0}w_0(\mathbf{L}_{\beta g}) = w_0(\mathbf{L}_{\beta_0 g})$ for any $\beta_0 \in \mathcal{W}_0^+$.
	\end{proposition}
	%
	We need a classical semigroup proposition as following:
	\begin{proposition}\label{Classical}
		Let $A$ be the generator of a semigroup $e^{tA}$, then:
		\begin{equation}
			w_0 (A) = \displaystyle\inf_{t>0} \frac{1}{t} \log \|e^{tA}\| = \frac{1}{t_0} \log r(e^{t_0A}) \ \text{for any } t_0 >0.
		\end{equation}
	\end{proposition}
	\begin{proof}[\textbf{proof}]
		See Proposition 2.2 in \cite[Chapter IV]{One-P}.
	\end{proof}
	\begin{proof}[\textbf{Proof of Proposition \ref{prop-2.3}}]  
		We denote $ \mathbf{L}_{\beta} := \mathbf{P}_\alpha + \beta \mathbf{L}_{g}' $.
		Let $\beta_0 \in \mathcal{W}_0^+$.
		We first show that $\displaystyle\lim_{\beta \to \beta_0}\|e^{t \mathbf{L}_{\beta}} - e^{t \mathbf{L}_{\beta_0}} \| =0 $.
		
		By $\|e^{t \mathbf{L}_{\beta}} \| \leq e^{Ct  \|g\|_{W^{1,\infty}} |\beta|} $ in Proposition \ref{prop-exist-semigroup} and $\|e^{t \mathbf{L}_{\beta}} \|_{L^2 \to H^1} \lesssim e^{Ct  \|g\|_{W^{1,\infty}} |\beta|} t^{-\frac{1}{\alpha}}$ in Proposition \ref{prop-regularity of e^L}, we have
		%
		\begin{equation}
			\begin{aligned}
				\|e^{t \mathbf{L}_{\beta}} - e^{t \mathbf{L}_{\beta_0}} \|
				&= \left\| \int_{0}^{t} e^{(t-s) \mathbf{L}_{\beta}} (\mathbf{L}_{\beta}-\mathbf{L}_{\beta_0}) e^{s \mathbf{L}_{\beta_0}} {\rm d}s \right\|
				\\ &\leq \int_{0}^{t} \| e^{(t-s) \mathbf{L}_{\beta}} \| \  \|\mathbf{L}_{\beta}-\mathbf{L}_{\beta_0}\|_{H^1 \to L^2} \|e^{s \mathbf{L}_{\beta_0}} \|_{L^2 \to H^1} {\rm d}s
				\\ &\lesssim |\beta - \beta_0| \int_{0}^{t} e^{(t-s) \|g\|_{W^{1,\infty}} |\beta| } \|\mathbf{L}_g' \|_{H^1 \to L^2} e^{s \|g\|_{W^{1,\infty}} |\beta_0| } s^{-\frac{1}{\alpha}} {\rm d}s
				\\ &\stackrel{\beta \to \beta_0}{\longrightarrow} 0,
			\end{aligned}
		\end{equation}
		thus $\displaystyle\lim_{\beta \to \beta_0}\|e^{t \mathbf{L}_{\beta}} - e^{t \mathbf{L}_{\beta_0}} \| =0 $.
		
		Let $\lambda \notin \sigma(e^{\mathbf{L}_{\beta_0}})$. Since
		\begin{equation}
			\begin{aligned}
				\left\| (\lambda {\rm Id} - e^{\mathbf{L}_{\beta}})^{-1} \right\|
				&\leq \| (\lambda {\rm Id}- e^{\mathbf{L}_{\beta}})^{-1} - (\lambda {\rm Id} - e^{\mathbf{L}_{\beta_0}})^{-1} \|
				+  \| (\lambda {\rm Id} - e^{\mathbf{L}_{\beta_0}})^{-1} \|
				\\ &\leq \| (\lambda {\rm Id} - e^{\mathbf{L}_{\beta}})^{-1} \| \ 
				\|e^{\mathbf{L}_{\beta}} - e^{\mathbf{L}_{\beta_0}} \| \ 
				\| (\lambda {\rm Id} - e^{\mathbf{L}_{\beta_0}})^{-1} \|
				+  \| (\lambda {\rm Id} - e^{\mathbf{L}_{\beta_0}})^{-1} \|,
			\end{aligned}
		\end{equation}
		we have that when $\|e^{\mathbf{L}_{\beta}} - e^{\mathbf{L}_{\beta_0}} \| 
		\leq \frac{1}{2\| (\lambda {\rm Id} - e^{\mathbf{L}_{\beta_0}})^{-1} \|}$:
		\begin{equation}
			\| (\lambda {\rm Id} - e^{\mathbf{L}_{\beta}})^{-1} \| 
			\leq 2 \| (\lambda {\rm Id} - e^{\mathbf{L}_{\beta_0}})^{-1} \|,
		\end{equation}
		and hence
		\begin{equation}\label{Rn-R0}
		 	\begin{aligned}
		 		\| (\lambda {\rm Id} - e^{\mathbf{L}_{\beta}})^{-1} - (\lambda {\rm Id} - e^{\mathbf{L}_{\beta_0}})^{-1} \|
		 		&\leq \| (\lambda {\rm Id} - e^{\mathbf{L}_{\beta}})^{-1} \| \ 
		 		\|e^{\mathbf{L}_{\beta}} - e^{\mathbf{L}_{\beta_0}} \| \ 
		 		\| (\lambda {\rm Id} - e^{\mathbf{L}_{\beta_0}})^{-1} \|
		 		\\ &\leq 2 \| (\lambda {\rm Id} - e^{\mathbf{L}_{\beta_0}})^{-1} \|^2  \    \|e^{\mathbf{L}_{\beta}} - e^{\mathbf{L}_{\beta_0}} \|
		 		\stackrel{\beta \to \beta_0}{\longrightarrow} 0.
		 	\end{aligned}
		\end{equation}
		
		By Proposition \ref{prop-w_0-L_u} we know that $\|e^{\mathbf{L}_{\beta}}\|_{ess} \leq e^{\frac{2\alpha-2-d}{2\alpha }} < 1$, so $\sigma (e^{\mathbf{L}_{\beta}}) \cap \{z_0 :|z_0| > e^{\frac{2\alpha-2-d}{2\alpha }}\} \in \sigma_d (e^{\mathbf{L}_{\beta}})$ for any $\beta$. 
		Now take any $z_0 \in \sigma (e^{\mathbf{L}_{\beta_0}}) \cap \{z :|z| \geq 1\}$ and
		any simple closed curve $\Gamma$ such that $\Gamma^\circ \cap \sigma(\mathbf{L}_{\beta_0}) = z_0$ ($\Gamma^\circ$ denotes the interior of curve $\Gamma$).
		By (\ref{Rn-R0}) we have
		\begin{equation}\label{PRn-PR0}
			\lim_{\beta \to \beta_0} \left\| \frac{1}{2\pi \text{i}} \int_\Gamma R(\lambda,e^{\mathbf{L}_{\beta}}) d\lambda
			- \frac{1}{2\pi \text{i}} \int_\Gamma R(\lambda,e^{\mathbf{L}_{\beta_0}}) d\lambda \right\| = 0,
		\end{equation}
		which indicates $\frac{1}{2\pi \text{i}} \int_\Gamma R(\lambda,e^{\mathbf{L}_{\beta}}) d\lambda \neq 0$ for all $\beta$ sufficiently close to $\beta_0$.
		So for those $\beta$ there exist $z_\beta \in \sigma(\mathbf{L}_{\beta}) \cap \Gamma^\circ$.
		Now take any sequence $\beta^{(n)} \to \beta_0$ and a sequence of simple closed curves $\{ \Gamma_n \}$ contracting to $z_0$ such that ${\rm dist} (z_0 , \Gamma_n) < 1-e^{\frac{2\alpha-2-d}{2\alpha }}$.
		The above argument ensures existence of a $ z_{\beta^{(n)}} \in \sigma_d(e^{\mathbf{L}_{\beta^{(n)}}}) \cap \Gamma_n^\circ$ for each $n$. Obvious $z_{\beta^{(n)}} \to z_0$.

		On one hand, 
		by $w_0(\mathbf{L}_{\beta_0}) \geq 0$ and Proposition \ref{prop-w_0-L_u} there exists $z' \in \sigma_d(\mathbf{L}_{\beta_0})$ such that ${\rm Re}(z') = w_0(\mathbf{L}_{\beta_0})$.
		Obviously $e^{z'} \in \sigma(e^{\mathbf{L}_{\beta_0}}) \cap \{ z: |z| \geq 1 \}$. Taking $z_0 = e^{z'}$ in last paragraph, for any $\beta \to \beta_0$, we will be able to find
		$z_\beta \in \sigma(e^{\mathbf{L}_{\beta}})$ with $z_\beta \to e^{z'}$.
		%
		Then by Proposition \ref{Classical} we have $w_0(\mathbf{L}_{\beta}) = \log r(e^{\mathbf{L}_{\beta}} ) \geq \log |z_\beta|$, hence $\displaystyle\varliminf_{\beta \to \beta_0}w_0(\mathbf{L}_{\beta}) \geq \log |e^{z'}| = w_0(\mathbf{L}_{\beta_0})$.

		On the other hand, we should prove that for any $\beta \to \beta_0$, $\displaystyle\varlimsup_{\beta \to \beta_0}w_0(\mathbf{L}_{\beta}) \leq w_0(\mathbf{L}_{\beta_0})$.
		Assume the opposite, one would be able to find a sequence $\beta^{(n)} \to \beta_0$ such that 
		$w_1 := \displaystyle\lim_{n \to \infty}w_0(\mathbf{L}_{\beta^{(n)}}) > w_0(\mathbf{L}_{\beta_0})$ and $w_0(\mathbf{L}_{\beta^{(n)}}) > w_0(\mathbf{L}_{\beta_0}) \geq 0$ for all $n$.
		By Proposition \ref{prop-w_0-L_u} and Proposition \ref{Classical} there exist  $z_n \in \sigma_d(e^{\mathbf{L}_{\beta^{(n)}}})$ such that $\log |z_n| = w_0(\mathbf{L}_{\beta^{(n)}}) \stackrel{n \to \infty}{\longrightarrow} w_1$ and $\log |z_n| > w_0(\mathbf{L}_{\beta_0})$.
		Note that $ |z_n| \leq \|e^{\mathbf{L}_{\beta^{(n)}}}\| \leq e^{\|g\|_{w^{1,\infty}} |\beta^{(n)}| }$, 
		so $\{z_n\} $ is precompact, there exists a subsequence $ \{{n_k} \} \in \mathbb{N}$ such that $z_{n_k} \to z'$.
		Arguing with (\ref{PRn-PR0}) similarly as above, we know that $z' \in \sigma(e^{\mathbf{L}_{{\beta_0}}})$,
		which means $\log |z'| = \displaystyle\lim_{k \to \infty}\log |z_{n_k}|  =w_1 > w_0(\mathbf{L}_{\beta_0})$. 
		But by Proposition \ref{Classical}, $w_0(\mathbf{L}_{\beta_0}) = \log r(e^{\mathbf{L}_{\beta_0}} ) \geq \log |z'| > w_0(\mathbf{L}_{\beta_0})$, a contradiction!

		Combining the two sides, we have $\displaystyle\lim_{\beta \to \beta_0}w_0(\mathbf{L}_{\beta}) = w_0(\mathbf{L}_{\beta_0})$ for $\beta_0 \in  \mathcal{W}_0^+$.
	\end{proof}
	Use Proposition \ref{prop-3.1-a} and Proposition \ref{prop-2.3} we can prove Proposition \ref{prop-3.1} easily.
	\begin{proof}[\textbf{Proof of Proposition \ref{prop-3.1}}]
		By Proposition \ref{prop-3.1-a} we know $1 \in \mathcal{W}_0^+$ and $ 0 \notin \mathcal{W}_0^+$. 
		Let ${w}_0^+ := \inf\{a : [a,1] \in \mathcal{W}_0^+ \}$. Note that $\mathcal{W}_0^+$ is an closed set, so $[{w}_0^+,1] \subset \mathcal{W}_0^+$.
		%
		Then by $w_0(\mathbf{L}_{{{w}_0^+} g}) = \displaystyle\lim_{\beta \uparrow {w}_0^+} w_0(\mathbf{L}_{\beta g}) \leq 0$, we have $w_0(\mathbf{L}_{{{w}_0^+} g})  =0$. Since by Proposition \ref{prop-2.3}, $\beta \mapsto w_0( \mathbf{L}_{\beta g})$ is continuous on $[0,1]$, then $[0,w_0( \mathbf{L}_{g})] \subset \{ w_0( \mathbf{L}_{\beta g}) : \beta \in \mathcal{W}_0^+ \}$.
		So by intermediate value theorem for continuous function, given any $A >0$, we can always choose an $a_0 \in \left(0,A \wedge w_0( \mathbf{L}_{g}) \right)$ such that $w_0(\mathbf{L}_{\beta_0 g}) = a_0$ for some $\beta_0 \in (0,1)$ .
	\end{proof}
	\section{Proof of the Main Theorems}
	In this section, we will prove our main result. 
	First we will show that the solutions to (\ref{FNSE_0}) and (\ref{Eq-v}) can be converted to each other.
	Then in the second subsection, we will prove Proposition \ref{prop-2.1}.
	Next we use Proposition \ref{prop-2.1} to prove Theorem \ref{Thm-1}, using the strategy in subsection \ref{strategy} in the third subsection.
	In the fourth and fifth subsection we will prove Theorem \ref{Thm-2} and \ref{Thm-1.1} respectively.
	
	\subsection{Relationship between (\ref{FNSE_0}) and (\ref{Eq-v})}\label{subs-eq}
	The solutions to (\ref{FNSE_0}) and (\ref{Eq-v}) can be transformed into each other. Specifically we have the following proposition:
	\begin{proposition}\label{prop-u-v}
		Let $u$ be a solution to (\ref{FNSE_0}) in sense of Definition \ref{def-1.1}, then ${v} := e^ {-\gamma {\rm B}_t + \frac{1}{2}\gamma^2 t} {u} $ is a solution to (\ref{Eq-v}) in sense of Definition \ref{def-1.2}.
		Moreover,
		let $ v$ be a solution to (\ref{Eq-v}) in sense of Definition \ref{def-1.2} and the following energy equality holds:
		\begin{equation}\label{Eq-energy-equ-v}
			\|v(t \wedge \rho) \|_{L^2}^2 - \|u_0\|_{L^2}^2  + 2 \int_{0}^{t \wedge \rho} \|v(s)\|_{\dot{H}^{\frac{\alpha}{2}}}^2 {\rm d}s 
			= 2 \int_{0}^{t \wedge \rho} \langle h(t) f,v \rangle {\rm d}s  \ \ \mathbb{P} - {\rm a.s.},
		\end{equation}
		then $ u := e^ {\gamma {\rm B}_t - \frac{1}{2}\gamma^2 t} v$ is a solution to (\ref{FNSE_0}) in sense of Definition \ref{def-1.1}.
	\end{proposition}
	\begin{proof}[\textbf{proof}]
		For any solution to (\ref{FNSE_0}) $u$ with a $\mathscr{F}_t$-stopping time $ \rho > 0 $, define $\widetilde{u}(t) := u(t\wedge \rho) $ is a $\mathscr{F}_t$-adopt process. For any $\psi \in C^\infty_c (\mathbb{R}^d)$, let $\widetilde{\psi} = \psi 1_{t \leq \rho}$, then:
		\begin{equation}
			\begin{aligned}
				\langle \widetilde{u}(t ),\widetilde{\psi} \rangle 
				=& \int_{0}^{t } \langle \widetilde{u}_i(s) \widetilde{u}_j(s), \partial_i \widetilde{\psi}_j \rangle {\rm d}s
				+ \int_{0}^{t } \langle \widetilde{u}(s),\Lambda^\alpha \widetilde{\psi} \rangle {\rm d}s 
				\\ &+ \int_{0}^{t} \langle f(s),\widetilde{\psi} \rangle {\rm d}s
				+ \int_{0}^{t} \langle \gamma \widetilde{u}(s),\widetilde{\psi} \rangle {\rm d}{\rm B}_s
				+ \langle u_0,\widetilde{\psi} \rangle,
			\end{aligned}
		\end{equation}
		so $\langle \widetilde{u}(t),\widetilde{\psi} \rangle $ is a semimartingale. Use It\^{o} formula to $e^ {\gamma {\rm B}_t - \frac{1}{2}\gamma^2 t} \langle \widetilde{u},\widetilde{\psi} \rangle$ we get $\widetilde{v} := v(t\wedge \rho)$ is a weak solution to (\ref{Eq-v}) immediately.
		
		Define a new stochastic process:
		\begin{equation}
			\begin{aligned}
				X(t) := 
				\underbrace{ \|u_0\|_{L^2}^2
					+ 2 \int_{0}^{t} \langle f,\widetilde{u} \rangle - \|\widetilde{u}(s)\|_{\dot{H}^{\frac{\alpha}{2}}}^2 {\rm d}s +\gamma^2 \int_{0}^{t} \|\widetilde{u}(s)\|_{L^2}^2 {\rm d} s }_{=: \|u_0\|_{L^2}^2 + \int_{0}^{t}A(s) {\rm d} s}
				+ \underbrace{ 2\gamma \int_{0}^{t} \|\widetilde{u}(s)\|_{L^2}^2 {\rm d} {\rm B}_s  }_{=: 2\gamma \int_{0}^{t} M(s) {\rm d} {\rm B}_s} .
			\end{aligned}
		\end{equation}
		Let $t <\rho$, by energy inequality $M(t) = \|u(t)\|_{L^2}^2 \leq X(t)$. By It\^{o} formula:
		\begin{equation}\label{Eq-energy-ine-X}
			\begin{aligned}
				e^{-2\gamma {\rm B}_t + \gamma^2 t} X(t) - \|u_0\|_{L^2}^2
				=& \int_{0}^{t} \gamma^2 e^{-2\gamma {\rm B}_s + \gamma^2 s} X(s) {\rm d} s
				+ \int_{0}^{t} -2\gamma e^{-2\gamma {\rm B}_s + \gamma^2 s} X(s) {\rm d} {\rm B}_s 
				+ \frac{1}{2} \int_{0}^{t} 4\gamma^2 e^{-2\gamma {\rm B}_s + \gamma^2 s} X(s) {\rm d} s
				\\ &+ \int_{0}^{t} e^{-2\gamma {\rm B}_s + \gamma^2 s} A(s){\rm d} s 
				+  \int_{0}^{t} e^{-2\gamma {\rm B}_s + \gamma^2 s} 2\gamma M(s) {\rm d} {\rm B}_s
				+ \int_{0}^{t} -4\gamma^2 e^{-2\gamma {\rm B}_s + \gamma^2 s}M(s){\rm d} s
				%
				%
				\\ \leq& \int_{0}^{t} -\gamma^2 e^{-2\gamma {\rm B}_s + \gamma^2 s}M(s){\rm d} s 
				+ \int_{0}^{t} e^{-2\gamma {\rm B}_s + \gamma^2 s} A(s){\rm d} s
				+  \int_{0}^{t} e^{-2\gamma {\rm B}_s + \gamma^2 s} 2\gamma \left(M(s)-X(s)\right) {\rm d} {\rm B}_s.
			\end{aligned}
		\end{equation}
		So we have
		\begin{equation}
			\begin{aligned}
				\mathbb{E} \|v(t \wedge \rho)\|_{L^2}^2 = \mathbb{E} e^{-2\gamma B_{t \wedge \rho} + \gamma^2 {t \wedge \rho}} M({t \wedge \rho}) 
				\leq& \mathbb{E} e^{-2\gamma B_{t \wedge \rho} + \gamma^2 {t \wedge \rho}} X({t \wedge \rho}) 
				\\ \leq& \mathbb{E} \int_{0}^{{t \wedge \rho}} e^{-2\gamma {\rm B}_s + \gamma^2 s} \left(A(s)- \gamma^2 M(s)\right){\rm d} s,
			\end{aligned}
		\end{equation}
		which is exactly (\ref{Eq-energy-ine-v}).
		So we prove ${v} := e^ {-\gamma {\rm B}_t + \frac{1}{2}\gamma^2 t} {u} $ is a solution to (\ref{Eq-v}) in sense of Definition \ref{def-1.2}.
		
		If $v$ is a solution to (\ref{Eq-v}) with equality (\ref{Eq-energy-equ-v}), 
		by similar way it is not hard to prove $u := e^ {\gamma {\rm B}_t - \frac{1}{2}\gamma^2 t} v $ is a solution to (\ref{FNSE_0}) in sense of Definition \ref{def-1.1}.
	\end{proof}

	\subsection{Proof of Proposition \ref{prop-2.1}} \label{New-way}
	This subsection primarily proves Proposition \ref{prop-2.1}.
	For that, we first establish a lemma concerning the existence of solutions to the equation in a general context.
	
	Consider the following equation:
	\begin{equation}\label{Eq-le-1}
		\left\{
		\begin{aligned}
			&\partial_t U = \mathbf{A} U + \mathbf{B}_t U,
			\\ &\lim_{t\to -\infty} \|U(t)\|_{H} = 0.
		\end{aligned}
		\right.
	\end{equation}
	Here $ \mathcal{H} $ is a Hilbert space, $ \mathcal{V} $ is a Banach space and $ \mathcal{V}'\subset \mathcal{H} \cap \mathcal{V}$ is a Banach space such that $\|\cdot\|_{\mathcal{H}} + \|\cdot\|_{\mathcal{V}} \leq \|\cdot\|_{\mathcal{V}'} $. $ \mathbf{A} $ is a densely defined linear operator on $H$, which generates a semigroup $e^{t\mathbf{A}}$ on $ H $, $\{\mathbf{B}_t;t\in \mathbb{R}\} $ is a family of continuous operators (generally non-linear) from $ \mathcal{V} $ to $\mathcal{H}$.
	
	\begin{lemma}\label{lemma-contract}
%
		Assume for some $ a_1 > a_0 > 0$,
		\\
		{\bf (1)}
		\begin{equation}\label{Eq-63}
			\begin{aligned}
				w_0(A) = a_0;
				\\ 
				\displaystyle\sup_{t \in (0,t_0)} \big( t^{\gamma}\| e^{t\mathbf{A}} \|_{\mathcal{H} \to \mathcal{V}'}\big) < \infty ,\ &for\  some \ t_0 >0 ,\gamma \in (0,1);
			\end{aligned}
		\end{equation}
		\\
		{\bf (2)}
		For all $ U, V \in \mathcal{U} := \{ U: \displaystyle\sup_{t<0}\|e^{-a_1 t} U(t)\|_{\mathcal{V}'} \leq 1 \} $ and $t < 0$,
		\begin{equation}\label{conditon-B}
			\begin{aligned}
				\| \mathbf{B}_t U(t) \|_{\mathcal{H}} &\leq M(t) e^{a_1 t},
				\\ \| \mathbf{B}_t U(t) - \mathbf{B}_t V(t) \|_{\mathcal{H}} &\leq M(t) \| U(t) - V(t) \|_{\mathcal{V}},
			\end{aligned}
		\end{equation}
		for some  
		continuous function $M(t),t \in \mathbb{R}$ s.t. $\displaystyle\limsup_{t\to -\infty} M(t) = 0$.

		Then there exists $ T_0 $ such that the equation (\ref{Eq-le-1}) has a solution $ u \in C \big( (-\infty, T_0); \mathcal{V}' \big) $ satisfying $ \displaystyle\sup_{t \leq T_0} e^{-a_1 t}\|U(t)\|_{\mathcal{V}'} < \infty  $.
		%
		Moreover, $T_0$ can be selected as $ \inf\{ t: \displaystyle\limsup_{\tau <t} M(\tau) \leq C_0\} $ where $C_0$ is a constant depends on spaces $\mathcal{H},\mathcal{V},\mathcal{V}'$ and constants $a_0, a_1,\gamma$.
	\end{lemma}
	\begin{proof}[\textbf{proof}]
		%
		For any fixed $ T > -\infty $, we define $ \| U \|_{X_T} := \displaystyle\sup_{t<T} \|e^{-a_1 t} u(t)\|_{\mathcal{V}'} $ and the set $ \mathcal{X}_T := \{ U: \|U\|_{X_T} \leq 1 \} $.
		
		Consider the mapping:
		\begin{equation}
			\begin{aligned}
				\Gamma : \mathcal{X}_T \to& \mathcal{X}_T,
				\\ U \mapsto& \int_{-\infty}^{t} e^{(t-s) \mathbf{A}}  \mathbf{B}_s U(s) {\rm d}s.
			\end{aligned}
		\end{equation}
		
		Then for $ U \in X_T $, by (\ref{Eq-63}) and (\ref{conditon-B}) we have:
		\begin{equation}
			\begin{aligned}
				\|\Gamma U\|_{X_T} \leq & \sup_{t <T} e^{-a_1 t} \int_{-\infty}^{t} \| e^{(t-s) \mathbf{A}} \mathbf{B}_s U(s) \|_{\mathcal{V}'}  {\rm d}s
				\\ \leq & \sup_{t <T}e^{-a_1 t} \int_{-\infty}^{t} \|e^{(t-s) \mathbf{A}}\|_{\mathcal{H} \to \mathcal{V}'} 
				\| \mathbf{B}_s U(s)\|_{\mathcal{H}} {\rm d}s
				\\ \leq & \sup_{t <T} \int_{-\infty}^{t} \|e^{(t-s) \mathbf{A}}\|_{\mathcal{H} \to \mathcal{V}'} 
				M(s)  e^{-a_1 (t-s)}{\rm d}s
				\\ \lesssim & \sup_{t <T} \int_{-\infty}^{t-t_0}  e^{(t-s-t_0) a_0} M(s) e^{-(t-s)a_1} {\rm d}s
				+ \sup_{t <T} \int_{t-t_0}^{t}  (t-s)^{-\gamma} M(s) e^{-(t-s)a_1} {\rm d}s 
				\\ \lesssim&  \sup_{t <T} M(t) \left(\int_{-\infty}^{T} e^{(T-s) (a_0-a_1) } e^{-t_0a_1} {\rm d}s 
				+ \int_{0}^{t_0} s^{-\gamma} {\rm d}s\right)
				\\ \lesssim& \sup_{t <T} M(t) \left(\frac{1}{a_1-a_0} + \frac{t_0^{1-\gamma}}{1-\gamma}\right)  .
			\end{aligned}
		\end{equation}
		
		Similarly, when $ u, v \in \mathcal{X}_T $, we can obtain:
		\begin{equation}
			\begin{aligned}
				\|\Gamma (U-V)\|_{X_T} \leq & \sup_{t <T} e^{-a_1 t} \int_{-\infty}^{t} \| e^{(t-s) \mathbf{A}} (\mathbf{B}_s U(s) - \mathbf{B}_s V(s) ) \|_{\mathcal{V}'}  {\rm d}s
				\\ \leq & \sup_{t <T}e^{-a_1 t} \int_{-\infty}^{t} \|e^{(t-s) \mathbf{A}}\|_{H \to \mathcal{V}'} 
				M(s)\| U(s)-V(s)\|_{\mathcal{H}} {\rm d}s
				\\ \lesssim& \sup_{t <T} M(t) \left(\frac{1}{a_1-a_0} + \frac{t_0^{1-\gamma}}{1-\gamma} \right) \|U-V\|_{X_T} .
			\end{aligned}
		\end{equation}
		%
		%
		
		Therefore, there exists a sufficiently small $ T $ such that $ \Gamma $ is a contraction mapping on $ \mathcal{X}_T \to \mathcal{X}_T $. By the fixed-point theorem, the equation (\ref{Eq-le-1}) has a fixed point $ u \in \mathcal{X}_T $. We denote such a sufficiently small $ T $ as $ T_0 $.
		%
		
		It is also necessary to prove that such solutions are continuous. It suffices to show that $ \Gamma(U) \in C\big( (-\infty, T); \mathcal{V}' \big) $ if $ u \in \mathcal{X}_T $.
		%
		
		First, we consider the case where $ t_n \downarrow \bar{t} $.
		\begin{equation}
			\begin{aligned}
				&\| \Gamma (U)(t_n) - \Gamma (U)(\bar{t}) \|_{\mathcal{V}'} 
				\\ \leq& \left\| \int_{-\infty}^{\bar{t}} (e^{(t_n -s)\mathbf{A} } - e^{(\bar{t} -s)\mathbf{A} }) \mathbf{B}u(s){\rm d}s \right\|_{\mathcal{V}'} + \left\| \int_{\bar{t}}^{t_n} e^{(t_n -s)\mathbf{A} } \mathbf{B}U(s){\rm d}s \right\|_{\mathcal{V}'}
				\\ =& \left\| \int_{-\infty}^{\bar{t}} e^{(\bar{t} -s)\mathbf{A} } (e^{(t_n -\bar{t})\mathbf{A} } -{\rm Id} ) \mathbf{B}u(s){\rm d}s \right\|_{\mathcal{V}'} 
				+ \left\| \int_{\bar{t}}^{t_n} e^{(t_n -s)\mathbf{A} } \mathbf{B}u(s){\rm d}s \right\|_{\mathcal{V}'}
				\\ \leq& \int_{-\infty}^{\bar{t}} \| e^{(\bar{t} -s)\mathbf{A} }\|_{\mathcal{H} \to \mathcal{V}'} \|e^{(t_n -\bar{t})\mathbf{A} } -{\rm Id} \|_{\mathcal{H} \to \mathcal{H}} \| \mathbf{B}u(s) \|_{\mathcal{H}} {\rm d}s 
				+ \int_{\bar{t}}^{t_n} \| e^{(t_n -s)\mathbf{A} } \|_{\mathcal{H} \to \mathcal{V}'} \| \mathbf{B}u(s) \|_{\mathcal{H}} {\rm d}s
				\\ \lesssim& \|e^{(t_n -\bar{t})\mathbf{A} } -{\rm Id} \|_{\mathcal{H} \to \mathcal{H}} M(\bar{t}) e^{a_1 \bar{t}}
				+ M(t_n) |t_n - \bar{t}|^{1-\gamma} e^{a_1 t_n}.
			\end{aligned}
		\end{equation}
		%
		%
		%
		Thus, the function $\Gamma(U)$ is right-continuous.
		Similarly, it can be shown to be left-continuous. Therefore, when $ U \in \mathcal{X}_T $, we have $ \Gamma(U) \in C\big( (-\infty, T_0); \mathcal{V}' \big) $. This implies that under the conditions of the lemma, there exists a solution in $ C\big( (-\infty, T_0); \mathcal{V}' \big) $.
	\end{proof}  
	
	Then we back to prove the Proposition \ref{prop-2.1}.
	\begin{proof}[\textbf{Proof of Proposition \ref{prop-2.1}}]
		Let $ \| U \|_{X_T} := \displaystyle\sup_{t<T} \|e^{-a_1 t} u(t)\|_{\mathcal{V}'} $ and the set $ \mathcal{X}_T := \{ U: \|U\|_X \leq 1 \} $. We will prove Proposition \ref{prop-2.1} by Lemma \ref{lemma-contract}.

		\textbf{The\ existence\ part:} For this , it is sufficient to verify the conditions of Lemma \ref{lemma-contract}.
		
		Now we denote $ Q_s(\cdot) = H(s) \mathbf{B}(\cdot,\cdot) + \widetilde{\mathbf{B}}(\cdot,\widetilde{U}(s)) + \widetilde{F}(s) $. For $ U, V \in X_T $,
		\begin{equation}
			\begin{aligned}
				\|Q_sU(s)\|_{\mathcal{H}} 
				\lesssim& H(s) \|U(s)\|_{\mathcal{V}}^2 +
				\|\widetilde{U}(s)\|_{\mathcal{V}} \|U(s)\|_{\mathcal{V}} 
				+ \|\widetilde{F}(s)\|_{\mathcal{H}}
				\\ \lesssim& H(s)e^{2 a_1 s}\|U\|_X^2 
				+ e^{\varepsilon s}  e^{a_1 s} \|U\|_X 
				+ e^{\varepsilon s} e^{a_1 s}
				\\ \lesssim& e^{\varepsilon s}e^{a_1 s}, \ s\in (-\infty ,T_0 \wedge 0),
			\end{aligned}			
		\end{equation}
		\begin{equation}
			\begin{aligned}
				\|Q_sU(s) - Q_sV(s)\|_{\mathcal{H}} 
				 \lesssim& \|U(s)-V(s)\|_{\mathcal{V}}(\|U(s)\|_{\mathcal{V}} + \|V(s)\|_{\mathcal{V}} )
				+ \|\widetilde{U}(s)\|_{\mathcal{V}} \|U(s)-V(s)\|_{\mathcal{V}}
				\\ \lesssim &  e^{ a_1 s}(\|U\|_X + \|V\|_X)\|U(s)-V(s)\|_{\mathcal{V}}
				+ e^{\varepsilon s} \|U(s)-V(s)\|_{\mathcal{V}} 
				\\ \lesssim & e^{\varepsilon s} \|U(s)-V(s)\|_{\mathcal{V}} , \ s\in (-\infty ,T_0 \wedge 0).
			\end{aligned}			
		\end{equation}
		%
		Take $ M(s) = Ce^{\varepsilon s} $ in Lemma \ref{lemma-contract}, then all conditions of Lemma \ref{lemma-contract} are satisfied.
		
		By Lemma \ref{lemma-contract} we can choose constant $C_0$ such that the equation (\ref{Eq-Normal}) has a solution $ U_r $ on $[0,T_0]$ where $T_0:= \inf\{ t : \displaystyle\sup_{\tau \leq t} \{ \| \widetilde{U}(\tau,\xi) \|_{\mathcal{V}'} + \|e^{-a_1 \tau} \widetilde{F}(\tau,\xi) \|_{\mathcal{H}} + |H(\tau)| \} \leq C_0 \} $
		and $U_r$ satisfies $ \|U_r(\tau)\|_{\mathcal{V}'} \leq M e^{a_1\tau} $.
%
%
		
		\textbf{The\ non-uniqueness\ part:}
		Let $ U_l(\tau) := \text{Re} (e^{z_0 \tau} \eta_0) $. It not hard to check 
		\begin{equation}
			\partial_\tau U_l = \mathbf{L}(U_l),
		\end{equation}
		and that $ \|e^{-a_0 \tau} U_l(\tau)\|_V $ is a periodic function in $\tau$ (note that $z_0 =a_0 + i b_0$).
		
		To get different solutions, we construct another solution $ U_l + U_p $ for the equation (\ref{Eq-Normal}). 
		This means that we have to solve
		\begin{equation}
			\begin{aligned}
				\partial_\tau U_p = \mathbf{L}(U_p) + H \mathbf{B}(U_p,\, U_p) + \widetilde{\mathbf{B}} (U_p ,\,  \widetilde{U}+H U_l ) + F_p,
			\end{aligned}			
		\end{equation}
		where $F_p(\tau) :=H(\tau) \mathbf{B}(U_l(\tau),\, U_l(\tau)) +\widetilde{\mathbf{B}}(U_l(\tau),\, \widetilde{U}(\tau)) +\widetilde{F}(\tau)$.
		
		The equation is in the same form as equation(\ref{Eq-Normal}) and we apply lemma \ref{lemma-contract} again to yield a solution $ U_p \in C_{(-\infty,T)} \mathcal{V}' $ such that $ \|U_p(\tau)\|_{\mathcal{V}'} \leq M e^{(a_0 +\frac{\varepsilon}{2})\tau} $.
		Verification of conditions is similar:
		\begin{equation}
			\begin{aligned}
				\|\widetilde{U}(\tau)+H(\tau)  U_l(\tau)\|_{\mathcal{V}} \leq \|\widetilde{U}(\tau)\|_{\mathcal{V}}+ H (\tau) \|U_l(\tau)\|_{\mathcal{V}}
				\lesssim e^{\varepsilon \tau} +  e^{a_0 \tau},
			\end{aligned}			
		\end{equation}
		\begin{equation}
			\begin{aligned}
				\|F_p(\tau)\|_{\mathcal{H}} &\leq H(\tau)\| \mathbf{B}(U_l(\tau),\, U_l(\tau))\|_{\mathcal{H}} + \| \widetilde{\mathbf{B}}(U_l(\tau),\, \widetilde{U}(\tau)) \|_{\mathcal{H}} + \|F(\tau)\|_{\mathcal{H}}
				\\ &\lesssim \|U_l(\tau)\|_{\mathcal{V}}^2 + \|U_l(\tau)\|_{\mathcal{V}} \|\widetilde{U}(\tau)\|_{\mathcal{V}}+ \|F(\tau)\|_{\mathcal{H}}
				\\ &\lesssim e^{2a_0 \tau} + C e^{(a_0+\varepsilon) \tau} + C e^{(a_1+\varepsilon) \tau} 
				%
				\lesssim e^{(a_0+\varepsilon) \tau}.
			\end{aligned}			
		\end{equation}

		%
		
		Now, both $ U_r $ and $ U_l + U_p $ are solutions to the equation (\ref{Eq-Normal}). So $ \displaystyle\lim_{\tau \to -\infty} e^{-a_0 \tau} U_r(\tau) =\displaystyle\lim_{\tau \to -\infty} e^{-a_0 \tau} U_p(\tau) = 0 $ while $ \displaystyle\limsup_{\tau \to -\infty}e^{-a_0 \tau} \|U_l(\tau)\|_{\mathcal{V}'} <\infty $, $ \displaystyle\lim_{\tau \to -\infty}e^{-a_0 \tau} U_l(\tau) $ does not exist, showing that $ U_r \neq U_l + U_p $.
		This completes the proof.
		%
		%
	\end{proof}

	\subsection{Proof of Theorem \ref{Thm-1}}\label{Proof of Thm-1}
	We introduce some notations in this subsection.
%
%
	We may assume without loss of generality that $p_0 \in (2,4)$ since $ \alpha \in \left( 1+ \frac{d}{4},1+ \frac{d}{2} \right)$.
	Denote $ \mathcal{H} = L^2_{\rm div} $,
	let $ \mathcal{V} := W^{1,p_0} \cap L^q , \ \mathcal{V}' := \mathcal{H} \cap \mathcal{V} \cap \dot{H}^{\frac{\alpha}{2} \vee 1} $
	equiped with the norm
	\begin{equation}
		\begin{aligned}
			\|\cdot\|_{\mathcal{V}'} := \|\cdot \|_{\mathcal{H}} + \| \cdot \|_{\mathcal{V}}+ \| \cdot \|_{\dot{H}^{\frac{\alpha}{2} \vee 1} },
		\end{aligned}		
	\end{equation}
	where $ \frac{1}{q} + \frac{1}{p_0 }= \frac{1}{2} $.
	%
	Note that $\| \mathbf{B}(U,V) \|_{\mathcal{H}} \leq \|U \|_{L^q} \|\nabla V\|_{L^{ p_0 }} \leq \|U\|_{\mathcal{V}}\|V\|_{\mathcal{V}} $.
	By (\ref{exp-tra}), we convert equation (\ref{FNSE_0}) into the random PDEs (\ref{Eq-v}).
	To prove Theorem \ref{Thm-1} we only need to show non-uniqueness of (\ref{Eq-v}) via Proposition \ref{prop-2.1}, by Proposition \ref{prop-u-v}.

	\begin{theorem}\label{Thm-4.1}
			Let $d \in \{2,3\},\alpha \in ( \frac{1}{2}+\frac{d}{4},1+\frac{d}{2} )$ , $ p_0>{\frac{d}{\alpha-1}} $,
			there exist a deterministic $f_0 \in L^1_t L^2_x (\mathbb{R}^+ \times \mathbb{R}^d) $,
			such that the equation (\ref{Eq-v}) has two local solutions $ v_1 $ and $ v_2 $ in sense of Definition (\ref{def-1.2}) which satisfy energy equality (\ref{Eq-energy-equ-v}) and $ v_1 \neq v_2 \ \mathbb{P} -{\rm a.s.}$ for any $\gamma \geq 0$, initial data $u_0 \in L^2_{\rm div} \cap L^{p_0}$ and force $f= f_0 +f_1$ where $\ \|f_1(t)\|_{L^2} = O(t^{\frac{2+d-4\alpha}{2\alpha} + \varepsilon_0}) \ {\rm as}\  t \ \to 0$.
		\end{theorem}
	Theorem \ref{Thm-1} follows directly from Theorem \ref{Thm-4.1} by Proposition \ref{prop-u-v}.
	We only prove Theorem \ref{Thm-4.1} here.
	We will find the following lemma useful.
	\begin{lemma}\label{lemma-stopping-B_t}
		$\rho := \inf \{ t >0: B(t) > t^{\frac{2+d-2\alpha}{4\alpha}} \} $
		is an almost surely positive $\mathscr{F}_t$-stopping time.
	\end{lemma}
	\begin{proof}[\textbf{proof}]
		Define $\mathscr{F}_t$-stopping time $\rho_n := \inf \{ t \geq \frac{1}{n}: B(t) > t^{\frac{2+d-2\alpha}{4\alpha}} \}$ and $\rho := \displaystyle\lim_{n\to \infty}\rho_n $.
		Since $ B(t) $ is continuous and $ \rho_{n+1} \leq \rho_n $, it follows that $ \rho $ is also a $\mathscr{F}_t$-stopping time.
		
		By the law of iterated logarithm:
		\[
		\mathbb{P} \left( \limsup_{t \downarrow 0} \frac{B(t)}{\sqrt{2t \ln \ln \frac{1}{t}}} = 1 \right)  =1,
		\]
		and ${t^{\frac{2+d-2\alpha}{4\alpha}} } \geq 2\sqrt{2t \ln \ln \frac{1}{t}}$ when $t$ is small enough,
		it immediately follows that:
		\begin{equation}
			\begin{aligned}
				\mathbb{P} \left(\rho =0\right) \leq \mathbb{P}  \left( \exists \, t_k \downarrow 0, \frac{B(t_k)}{t_k^{\frac{2+d-2\alpha}{4\alpha}} } \geq 1 \right)
				\leq \mathbb{P} \left(  \exists \, t_k \downarrow 0,\frac{B(t_k)}{\sqrt{2t_k \ln \ln \frac{1}{t_k}}} >2 \right)
				=0 .
			\end{aligned}
		\end{equation}
	\end{proof}
	%
	\begin{proof}[\textbf{Proof of Theorem \ref{Thm-4.1}}]
		We follow the idea in the section \ref{strategy}.
		Subsequently, we seek a solution to (\ref{Eq-varpi'}) in the following form:
		\begin{equation}
			\mathsf{W}(\tau,\xi) = \bar{U}(\xi) +  U_r(\tau,\xi).
		\end{equation}
		
		Now we choose some suitable $\bar{U}(\xi)$.
		Define
		\begin{equation}
			\delta_0 = \min \{ \frac{2+d-2\alpha}{4\alpha}, (1- \frac{p_0+d}{\alpha p_0}) , \varepsilon_0 \} > 0.
		\end{equation}
		Let the constant $A$ in Proposition \ref{prop-3.1} be $\frac{\delta_0}{2}$, by Proposition \ref{prop-3.1} there exists $a_0 \in (0, \frac{\delta_0}{2})$ and $\bar{U}(\xi)$ s.t. (1)(2) and (3) in Proposition \ref{prop-3.1} are satisfied.
		
		Since $\bar{U}(\xi)$ is determined, we only need to consider $U_r$.
%
		In the following, we solve:
		\begin{equation}\label{Eq-U_r-1}
			\left\{
			\begin{aligned}
				&\partial_\tau U_r 
				= \mathbf{L}_{\bar{U}}( U_r) - H \mathbf{B}( U_r,\, U_r ) - \widetilde{\mathbf{B}} \big(  U_r,\, \widetilde{U}  \big)
				+ \widetilde{F} ,
				\\ & \nabla_\xi \cdot  U_r =0,
				\\ & \lim_{\tau \to -\infty} e^{\frac{2+d-2\alpha}{2\alpha} \tau } \|U_r(\tau)\|_{L^2}=0,
			\end{aligned}
			\right.
		\end{equation}
		where 
		\begin{equation*}
			\begin{aligned}
				\widetilde{U} &= (H-1)\bar{U}+H\mathring{U},
				\\ \widetilde{F} &= F_r - (H-1) \mathbf{P} (\bar{U} \cdot \nabla \bar{U} ) -H \mathbf{P}(\mathring{U} \cdot \nabla \bar{U} + \bar{U} \cdot \nabla \mathring{U} ) - H \mathbf{P}(\mathring{U} \cdot \nabla\mathring{U}) .
			\end{aligned}
		\end{equation*}
		
		We should use Proposition \ref{prop-2.1} to show the non-uniqueness of (\ref{Eq-U_r-1}) and hence of (\ref{Eq-varpi'}).
		%
		From Proposition \ref{prop-3.1} and remark \ref{remark-eta-V}, we already know that conditions (A1) and (A3) of Proposition \ref{prop-2.1} are satisfied, so it remains only to verify condition (A2).
		
		First we show some estimate on $\mathring{U}$.
		By the coordinate transformation (\ref{tra-1}) (\ref{tra-2}) and standard heat estimate:
		\begin{equation}
			\begin{aligned}
				&\| \mathring{U}(\tau)\|_{L^p_\xi} = e^{(1- \frac{p+d}{\alpha p}) \tau} \|\mathring{u}(t)\|_{L^p_x} \lesssim e^{ (1- \frac{p+d}{\alpha p}) \tau} \|u_0\|_{L^p_x},
				\\ &\| \nabla_\xi \mathring{U}(\tau)\|_{L^p_\xi} = e^{ (1- \frac{d}{\alpha p}) \tau} \|\nabla_x \mathring{u}(t)\|_{L^p_x} 
				\lesssim e^{ (1- \frac{d}{\alpha p}) \tau} e^{- \frac{1}{\alpha} \tau} \|u_0\|_{L^p_x}
				= e^{ (1- \frac{p+d}{\alpha p}) \tau} \|u_0\|_{L^p_x}.
			\end{aligned}			
		\end{equation}
		%
		
		Define $\mathscr{F}_t$-stopping time
		\begin{equation}
			\rho' := \inf \{ | \gamma B_t - \frac{1}{2} \gamma^2 t | \geq 1 \}.
		\end{equation}
		Recall $h(t) =e^ {-\gamma {\rm B}_t +\frac{1}{2}\gamma^2 t}$, $H(\tau) := h(e^\tau)^{-1}$, by $|e^x-1| \leq 3|x| $ for $x \leq 1$, we have the following inequality for any $\tau \leq \ln \rho \wedge \ln \rho'$:
		\begin{equation}
			\begin{aligned}
				|H(\tau) -1| + |H(\tau)^{-1} -1| \leq 6 |-\gamma {\rm B}_{e^\tau} +\frac{1}{2}\gamma^2 e^\tau|
				\leq 6 \gamma e^{\frac{2+d-2\alpha}{4\alpha} \tau} + 4\gamma^2 e^\tau
				\lesssim e^{\frac{2+d-2\alpha}{4\alpha} \tau},
			\end{aligned}
		\end{equation}
		 where $\rho$ is the $\mathscr{F}_t$-stopping time in Lemma \ref{lemma-stopping-B_t}
		
		Since $\delta_0 = \min \{ \frac{2+d-2\alpha}{4\alpha}, (1- \frac{p_0+d}{\alpha p_0}) , \varepsilon_0 \} > 0$, $0< a_0 < \frac{\delta_0}{2}$, we have that (for $\tau \leq \ln \rho \wedge \ln \rho'$):
		\begin{equation}
			\begin{aligned}
				\|(H-1)\bar{U}+H\mathring{U} \|_{\mathcal{V}} 
				&\leq |H-1| \ \|\bar{U}\|_{\mathcal{V}} + H \|\mathring{U}(\tau) \|_{\mathcal{V}} 
				\\ &\lesssim |H-1| + He^{ (1- \frac{p_0+d}{\alpha p_0}) \tau} \|u_0\|_{L^{p_0}_x}
				\\ &\lesssim e^{ \delta_0 \tau},
			\end{aligned}		 
		\end{equation}
		%
		\begin{equation}
			\begin{aligned}
				\|\widetilde{F}\|_{\mathcal{H}} &= \| (H^{-1}-1)F - (H-1) \mathbf{B} (\bar{U},\,  \bar{U} ) -H \widetilde{\mathbf{B}}(\mathring{U} ,\, \bar{U} ) - H \mathbf{B}(\mathring{U},\, \mathring{U}) \|_{\mathcal{H}}
				\\ &\leq |H^{-1}-1| \ \|F_0\|_{\mathcal{H}} + |H^{-1}| \ \|F_1\|_{\mathcal{H}} + |H-1| \ \|\bar{U}\|_{\mathcal{V}}^2 + 2H \|\bar{U}\|_{\mathcal{V}}\|\mathring{U}\|_{\mathcal{V}} + H\|\mathring{U}\|_{\mathcal{V}}^2
				\\ &\lesssim |H^{-1}-1|+ e^{\varepsilon_0 \tau}  + |H-1| +C e^{ (1- \frac{p_0+d}{\alpha p_0}) \tau}
				\\ &\lesssim e^{ (a_1 + \frac{\delta_0}{2}) \tau}.
			\end{aligned}		 
		\end{equation}
		Let $\varepsilon = \frac{\delta_0}{2}$ in condition (A2) Proposition \ref{prop-2.1}, then condition (A2) is satisfied. 
		
		Now we can use Proposition \ref{prop-2.1}, leading to existence of two distinct local solutions, $ U_r $ and $ U_r' $, to equation (\ref{Eq-U_r-1}) on $(-\infty,T_0]$ such that $U_r, U_r' \in C( (-\infty,T_0) ; \mathcal{V}')$, $ \displaystyle\sup_{\tau < T_0,i=1,2}\{ \| e^{-a_0 \tau} U_i(\tau, \xi) \|_{\mathcal{V}'} \} < \infty $ where $$T_0:= \min \left\{  \inf\{ t : \displaystyle\sup_{\tau \leq t} \{ \| \widetilde{U}(\tau,\xi) \|_{\mathcal{V}} + \|e^{-a_1 \tau} \widetilde{F}(\tau,\xi) \|_{\mathcal{H}} + |H(\tau)| \} \leq C_0 \} , \ln \rho, \ln \rho' \right\}.$$
		Correspondingly, this gives rise to two distinct solutions to (\ref{Eq-varpi'}) by $ \mathsf{W}_1 = \bar{U} + U_r $ and $ \mathsf{W}_2 = \bar{U} + U_r' $.
		
		Define $\mathscr{F}$-stopping time
		\begin{equation}
			\begin{aligned}
				t_0 :=& e^{T_0},
				\\ t_0' :=& \inf \{e^{\tau} : \|U_r(\tau)\|_{\dot{H}^\alpha_\xi}+\|U_r'(\tau)\|_{\dot{H}^\alpha_\xi} > 2 \|\bar{U}\|_{\dot{H}^\alpha_\xi} \},
				\\ t_1 :=& t_0 \wedge t_0'.
			\end{aligned}
		\end{equation}
		Then $t_1 \leq \rho \wedge \rho' $ is an almost surely positive $\mathscr{F}_t$-stopping time.
%
%
		%
%
		By appealing to coordinate transformation, we obtain two solutions to (\ref{Eq-v}) on $[0,t_1]$:
		\begin{equation}\label{sca-the}
			v_i(t,x) = e^{(\frac{-1}{\alpha}+1) \tau} \left(\mathsf{W}_i(\tau,\xi) + \mathring{U}(\tau,\xi)\right) ,\ \ \tau = \log t,  \ \xi = \frac{x}{t^{ \frac{1}{\alpha} }}.
		\end{equation}
		
%
%
		
		To this end, we verify Definition \ref{def-1.2} for $v_1,v_2$.
		First, we justify the continuity of solutions. 
		%
		Clearly $ v_1, v_2 \in C((0,t_1); L^2_x) $ since $ \mathsf{W}_i,\mathring{U} \in C((-\infty,T_0); \mathcal{V}') $.
		Actually we have $ v_1, v_2 \in C((0,t_1); H^1_x \cap L^p), \ p \in (2, \frac{2d}{2+d-2\alpha}) $.
		Moreover, by $ \displaystyle\lim_{\tau \to -\infty} e^{\frac{2+d-2\alpha}{2\alpha} \tau } \left(\| U_r\| +  \|U_r'\|+\| \bar{U}\|\right) = 0 $ and $ \mathring{u} \in C([0, t_1); L^2) $, it follows that $ v_1, v_2 \in C([0, t_1); L^2) $.
		
		Next, we show integrability. Recall $t_1 \leq \inf \{e^{\tau} : \|U_r(\tau)\|_{\dot{H}^\alpha_\xi}+\|U_r'(\tau)\|_{\dot{H}^\alpha_\xi} > 2 \|\bar{U}\|_{\dot{H}^\alpha_\xi} \}$, we have:
		\begin{equation}
			\begin{aligned}
				\|{v_i}\|_{L^2( 0,t_1; \dot{H}^\alpha_x ) }
				\leq& \|{\mathring{u}}\|_{L^2( 0,t_1; \dot{H}^\alpha_x ) }
				+   \left(\int_{-\infty}^{\ln{t_1}} \left(\|e^{ \frac{d-2\alpha}{2 \alpha} \tau} \bar{U}\|_{ \dot{H}^\alpha_\xi }  
				+\|e^{ \frac{d-2\alpha}{2 \alpha} \tau}  U_r(\tau)\|_{\dot{H}^\alpha_\xi } + \|e^{ \frac{d-2\alpha}{2 \alpha}  \tau} U_r'(\tau) \|_{\dot{H}^\alpha_\xi } \right)^2 {\rm d}\tau \right)^{\frac{1}{2}}
				\\ \leq& \|u_0\|_{L^2}
				+3  \|{\bar{U}}\|_{\dot{H}^\alpha_\xi} \left(\frac{\alpha}{d-\alpha}\right)^{\frac{1}{2}} t_1^{\frac{d-\alpha}{2\alpha}}
				< \infty,
			\end{aligned}
		\end{equation}
		which means $v_1,v_2 \in L^2_t \dot{H}^\alpha_x$ since $v_i \in C((0,t_1) ; \dot{H}^{\frac{\alpha}{2}}_x)$.
		Since $ \|f(t)\|_{L^2_x(\mathbb{R}^d)} \leq t^{-2 + \frac{1}{\alpha}} \|F(\tau)\|_{L^2_\xi} t^{\frac{d}{2\alpha}} \leq C t^{\frac{2+d -4 \alpha}{2\alpha}} $, it follows that $ \|f\|_{L^1 (0,T;  L^2_x )} < \infty $ for any $T<\infty$. So $v_1,v_2 \in L^2(0,t_1; L^2(\mathbb{R}^d))$, $f \in L^1_tL^2_x$.
		
		Finally, $ v_1,v_2 $ satisfies the energy equality on $ [0,t_1] $ obviously, since $ v_i \in C\big((0,t_1); H^{1}_x\big) $.
%
%
%
%
%
	\end{proof}

	\subsection{Proof of Theorem \ref{Thm-2}}
	\begin{proof}[\textbf{proof}]
		We follow the proof of Theorem \ref{Thm-1} to get one solution firstly.
		Let ${v} := e^ {\gamma {\rm B}_t - \frac{1}{2}\gamma^2 t} {u}$, so
		(\ref{FNSE_0}) can be rearranged  as (\ref{Eq-v}). By Proposition \ref{prop-u-v} we only need to prove the uniqueness of (\ref{Eq-v}).
		Let $ \mathring{u} $ be the solution to the following heat equation:
		\begin{equation}
			\left\{
			\begin{aligned}
				&\partial_t \mathring{u} + \Lambda^\alpha \mathring{u}= 0,
				\\ &\nabla \cdot \mathring{u} = 0,
				\\ &\mathring{u}|_{t=0} = u_0.
			\end{aligned}
			\right.
		\end{equation}
		We also seek a solution to (\ref{Eq-v}) of the following form:
		\begin{equation}
			v(t,x) = \mathring{u}(t,x) + u_r(t,x).
		\end{equation}
		
		Applying the coordinate transformation (\ref{tra-1}) (\ref{tra-2}), we are led to equations:
%
		\begin{equation}\label{NS_2}
			\left\{
			\begin{aligned}
				&\partial_\tau U_r -\frac{\alpha-1}{\alpha} U_r - \frac{1}{\alpha} \xi \cdot \nabla_\xi U_r + H \mathbf{P}(U_r \cdot \nabla U_r + \mathring{U} \cdot \nabla U_r +U_r \cdot \nabla \mathring{U}) 
				\\ &= \Lambda^\alpha U_r + F - H \mathbf{P}(\mathring{U} \cdot \nabla\mathring{U}) ,
				\\ & \nabla_\xi \cdot U_r =0,
				\\ & \lim_{\tau \to -\infty} e^{\frac{2+d-2\alpha}{2\alpha} \tau } \|U_r(\tau)\|_{L^2}=0,
			\end{aligned}
			\right.
		\end{equation}
		where $ \displaystyle\lim_{\tau \to -\infty}\| e^{\varepsilon \tau} F(\tau) \|_{L^2} =0$.
		
		Let $ H := L^2_{\rm div}$, $ V := W^{1,p_0} \cap L^q , \ \mathcal{V}' := H \cap V \cap \dot{H}^{\frac{\alpha}{2} \vee 1} $
		where $ \frac{1}{q} + \frac{1}{p_0 }= \frac{1}{2} $ (the spaces here are the same as those in proof of Theorem \ref{Thm-1}).
		By Proposition \ref{lemma-regularity-e^tP} we know the condition (A1) in the Proposition \ref{prop-2.1} is satisfied for any $a_0 >0$.
		
		Let $a_1 = \frac{\varepsilon}{2}$, $a_0 := \frac{\varepsilon}{4}$, where $\varepsilon$ is the constant in Theorem \ref{Thm-2}.
		Then by the Proposition \ref{prop-2.1}, (\ref{NS_2}) has a solution $ U \in C_\tau \mathcal{V}'_\xi $ such that $ \| U(\tau) \|_{\mathcal{V}'} \lesssim e^{\frac{1}{2} \varepsilon \tau} $,
		so (\ref{Eq-v}) has a solution $ v \in L^\infty(0,T ; L^2_x) \cap L^2(0,T ; H^{\frac{\alpha}{2}}_x)\  \mathbb{P}$-a.s. and $v$ satisfy following energy equality:
		\begin{equation}\label{Eq-uni-energy-equ-v}
			\mathbb{E} \|v(t \wedge \rho) \|_{L^2}^2 - \|u_0\|_{L^2}^2  + 2 \mathbb{E} \int_{0}^{t \wedge \rho} \|v(s)\|_{\dot{H}^{\frac{\alpha}{2}}}^2 {\rm d}s 
			= 2 \mathbb{E} \int_{0}^{t \wedge \rho} \langle h(t) f,v \rangle {\rm d}s .
		\end{equation}
		
		Moreover, since $ \| \nabla_x u(t) \|_{\dot{H}^1_x} = t^{\frac{d-2\alpha}{2\alpha}} \| \nabla_\xi U(\ln t) \|_{\dot{H}^1_\xi}  $, we can define stopping time $\rho'$ such that $ \| \nabla_x u(t) \|_{\dot{H}^1_x} \leq t^{\frac{d-2\alpha}{2\alpha} + \frac{\varepsilon}{4}} $ and $\mathbb{P} \left(\rho' >0\right) =1$.
		
		
%
%
		
		Note that by H\"{o}lder's inequality and Young's inequality, we havet
		\begin{equation}
			\begin{aligned}
				&| \langle v \cdot \nabla u, w \rangle | \leq \|v\|_{L^4} \|w\|_{L^4} \|\nabla u\|_{L^2}
				\leq C \|v\|_{H^{\frac{d}{4}}} \|w\|_{H^{\frac{d}{4}}} \|\nabla u\|_{L^2}
				\\ \leq& C \|v\|_{H^{\frac{\alpha}{2}}}^{\frac{d}{2\alpha}} \|v\|_{L^2}^{\frac{2\alpha-d}{2\alpha}} \|w\|_{H^{\frac{\alpha}{2}}}^{\frac{d}{2\alpha}} \|w\|_{L^2}^{\frac{2\alpha-d}{2\alpha}}  \|\nabla u\|_{L^2}
				\leq \|v\|_{H^{\frac{\alpha}{2}}}\|w\|_{H^{\frac{\alpha}{2}}} + C \|v\|_{L^2} \|w\|_{L^2}  \|\nabla u\|_{L^2}^{\frac{2\alpha}{2\alpha-d}}.
			\end{aligned}
		\end{equation}
		Then for the solution $v$ we get form Proposition \ref{prop-2.1} and any solution $v'$ to (\ref{Eq-v}) we have:
		\begin{equation}
			\begin{aligned}
				\int_{0}^{t \wedge \rho'} |\langle v(s)\cdot \nabla v(s), v'(s) \rangle| {\rm d}s 
				\lesssim \|v\|_{L^2_t H_x^{\frac{\alpha}{2}}}^2 + \|v'\|_{L^2_t H_x^{\frac{\alpha}{2}}}^2
				+ \|v\|_{L^\infty L^2_x}\|v'\|_{L^\infty L^2_x} \|\nabla v\|_{L^{\frac{2\alpha}{2\alpha -d}}_t L^2_x}^{\frac{2\alpha-d}{2\alpha}} < \infty,
				\\ \int_{0}^{t \wedge \rho'} | \langle v'(s)\cdot \nabla v'(s), v(s) \rangle | {\rm d}s 
				=  \int_{0}^{t \wedge \rho'} | \langle v'_i(s) v_j'(s) \partial_i v_j(s) \rangle | {\rm d}s 
				\lesssim \|v'\|_{H^{\frac{\alpha}{2}}}^{2} + \|v'\|_{L^2}^{2}  \|\nabla v\|_{L^2}^{\frac{2\alpha}{2\alpha-d}}
				< \infty.
			\end{aligned}			 
		\end{equation}
		It is not hard to verify that the regularities of $v$ and $v'$ ensure the validity of the following equalities (in almost surely sense):
		\begin{equation}\label{Eq-er-v,v'}
			\begin{aligned}
				\\ \langle v(t \wedge \rho'),v'(t \wedge \rho') \rangle -  \langle u_0,u_0 \rangle
				=& \int_{0}^{t \wedge \rho'} \langle v(s), \partial_s v'(s) \rangle {\rm d}s
				+ \int_{0}^{t \wedge \rho'} h(s)^{-1} \langle v_i(s) v_j(s), \partial_i v'_j \rangle {\rm d}s
				\\ &+ \int_{0}^{t \wedge \rho'} \langle v(s),\Lambda^\alpha v'(s) \rangle {\rm d}s 
				+ \int_{0}^{t \wedge \rho'} \langle h(s)f(s),v'(s) \rangle {\rm d}s,
			\end{aligned}
		\end{equation}
		\begin{equation}\label{Eq-er-v',v}
			\begin{aligned}
				\\ \langle v'(t \wedge \rho'),v(t \wedge \rho') \rangle -  \langle u_0,u_0 \rangle
				=& \int_{0}^{t \wedge \rho'} \langle  v'(s), \partial_s v(s) \rangle {\rm d}s
				+ \int_{0}^{t \wedge \rho'} h(s)^{-1} \langle v'_i(s) v'_j(s), \partial_i v_j \rangle {\rm d}s
				\\ &+ \int_{0}^{t \wedge \rho'} \langle v'(s),\Lambda^\alpha v(s) \rangle {\rm d}s 
				+ \int_{0}^{t \wedge \rho'} \langle h(s)f(s),v(s) \rangle {\rm d}s.
			\end{aligned}
		\end{equation}
		
		Let $w := v-v'$. Add (\ref{Eq-er-v,v'}) and (\ref{Eq-er-v',v}), by $v$ satisfy energy equality (\ref{Eq-uni-energy-equ-v}) and $ v'$ satisfy energy inequality (\ref{Eq-energy-ine-v}), we obtain that:
		\begin{equation}\label{Eq-energy-w}
			\begin{aligned}
				\mathbb{E} \|w(t \wedge \rho') \|^2 =& \mathbb{E} \|v(t \wedge \rho') \|^2 + \mathbb{E} \|v'(t \wedge \rho') \|^2 -2 \mathbb{E} \langle v'(t \wedge \rho'),v(t \wedge \rho') \rangle
				\\ =& \mathbb{E} \|v(t \wedge \rho') \|^2 + \mathbb{E} \|v'(t \wedge \rho') \|^2 
				-2 \|u_0\|^2 - \mathbb{E} \int_{0}^{t \wedge \rho'} h(s)^{-1} \langle w(s) \cdot \nabla w(s), u(s) \rangle {\rm d}s
				\\ &+ \mathbb{E} \int_{0}^{t \wedge \rho'} \left( \|v(s)\|_{\dot{H}^{\frac{\alpha}{2}}}^2 + \|v'(s)\|_{\dot{H}^{\frac{\alpha}{2}}}^2 -\|w(s)\|_{\dot{H}^{\frac{\alpha}{2}}}^2 \right) {\rm d}s 
				- E\int_{0}^{t \wedge \rho'} \langle h(s)f(s),v(s)+v'(s) \rangle {\rm d}s 
				\\ \leq& - \mathbb{E} \int_{0}^{t \wedge \rho'} h(s)^{-1} \langle w(s) \cdot \nabla w(s), u(s) \rangle {\rm d}s
				- \mathbb{E} \int_{0}^{t \wedge \rho'}  \|w(s)\|_{\dot{H}^{\frac{\alpha}{2}}}^2  {\rm d}s
				\\ \leq& \mathbb{E} \int_{0}^{t \wedge \rho'} \left(\|w(s)\|_{H^{\frac{\alpha}{2}}}^{2}
				+ C \|w(s)\|_{L^2}^{2}  \|\nabla v(s)\|_{L^2}^{\frac{2\alpha}{2\alpha-d}} 
				- \|w(s)\|_{H^{\frac{\alpha}{2}}}^{2}\right) {\rm d}s
				\\ \leq& C \sup_{s<t} \mathbb{E} \|w(s \wedge \rho')\|_{ L^2_x } 
				\int_{0}^{t} \mathop{\mathrm{ess~sup}}_{\omega \in \Omega} 
				\left\{ \|v(s \wedge \rho')\|_{\dot{H}^1}^{\frac{2\alpha}{2\alpha-d}}  \right\} {\rm d}s.
				\\ \leq& C \sup_{s<t} \mathbb{E} \|w(s \wedge \rho')\|_{ L^2_x } \frac{2(2\alpha -d)}{\varepsilon \alpha} t^{\frac{\varepsilon \alpha}{2(2\alpha -d)}}.
			\end{aligned}
		\end{equation}
		Since $\frac{\varepsilon \alpha}{2(2\alpha -d)} > 0$, in (\ref{Eq-energy-w}) we can choose $t_1 >0$ small enough such that $C\frac{2(2\alpha -d)}{\varepsilon \alpha} t_1^{\frac{\varepsilon \alpha}{2(2\alpha -d)}}<1$. Then we have
		\begin{equation}
			\sup_{s<t_1} \mathbb{E} \|w(s \wedge \rho')\|_{ L^2_x } \leq \left(C \frac{2(2\alpha -d)}{\varepsilon \alpha} t_1^{\frac{\varepsilon \alpha}{2(2\alpha -d)}} \right) \sup_{s<t_1} \mathbb{E} \|w(s \wedge \rho')\|_{ L^2_x },
		\end{equation}
		which is only possible when $\sup_{s<t_1} \mathbb{E} \|w(s \wedge \rho')\|_{ L^2_x } =0$.
		This shows that the solution is unique up to $\rho' \wedge t_1$.
%
	\end{proof}

	\subsection{Proof of Theorem \ref{Thm-1.1}}
%
	\begin{proof}[\textbf{proof}]
		We aim to show the non-uniqueness of equation (\ref{Eq-varpi'}) for $\mathring{U} =0$, i.e.
		\begin{equation}\label{Eq-varpi''}
			\begin{aligned}
				\partial_\tau U   =\mathbf{P}_\alpha U - H\mathbf{B}(U,U) + H^{-1}F.
			\end{aligned}
		\end{equation}
		We choose $\bar{U} = \bar{U}(\xi)$ by Proposition \ref{prop-3.1} with any $A>0$ and consider following linear equation for $U_\ell$:
		\begin{equation}\label{eq-th2-l}
			\begin{aligned}
				\partial_\tau U_\ell   =\mathbf{L}_{\bar{U}} U_\ell - (H-1) \widetilde{\mathbf{B}}(\bar{U},\, U_\ell)  .
			\end{aligned}
		\end{equation}
		%
		Applying Proposition \ref{prop-2.1} yields a non-trivial local solution to (\ref{eq-th2-l}) such that $\displaystyle\lim_{\tau \to -\infty} e^{\frac{2+d-2\alpha}{2\alpha} \tau } \|{U_\ell}(\tau)\|_{H^1}=0 $. It is clear that the solution can be uniquely extended so that $ U_\ell \in C(( -\infty,\infty); H^{1} ) $ as equation (\ref{eq-th2-l}) is homogeneous.
		Actually we can find a $U_\ell$ satisfy $ U_\ell \in C(( -\infty,\infty); H^{N} ) $ for any $N >0$, since $\bar{U} \in C^\infty_c$, $\mathbf{L}_{\bar{U}}$ generate a semigroup on $H^N$ and $\|e^{t\mathbf{L}_{\bar{U}}}\|_{H^N\to H^{N+1}} \leq C t^{\frac{1}{\alpha}}$ by Proposition \ref{prop-exist-semigroup} and Remark \ref{re-e^L-H^N}.
		
		
		%
%
		%
		Then it can be verified that $U_1 := \bar{U} - \frac{1}{2} U_\ell $ satisfies (\ref{Eq-varpi''}) with
		\begin{equation}
			\begin{aligned}
				F :=& H \Big( \partial_\tau U_1 - \mathbf{P}_\alpha U_1 + H \mathbf{B}(U_1,\, U_1)\Big)
				\\ =& H \left(-\frac{1}{2}\mathbf{L}_{\bar{U}} U_\ell + \frac{1}{2}(H-1) \widetilde{\mathbf{B}}(\bar{U},\, U_\ell) - \mathbf{P}_\alpha U_1 + H \mathbf{B}(U_1,\, U_1) \right) 
				\\ =& H \left(- \mathbf{P}_\alpha \bar{U} + H \mathbf{B}(\bar{U},\, \bar{U}) + \frac{1}{4} H \mathbf{B}(U_\ell,\, U_\ell)\right).
			\end{aligned}
		\end{equation}
%
		To this end one can see $U_2 := U_1 + U_\ell$ is another solution of (\ref{Eq-varpi''}):
		\begin{equation}\label{check}
			\begin{aligned}
				&\mathbf{P}_\alpha U_2 - H \mathbf{B}(U_2,\, U_2) + F
				\\ =& \mathbf{P}_\alpha U_1 +\mathbf{P}_\alpha U_\ell - H\mathbf{B}(U_1,\, U_1) - H \widetilde{\mathbf{B}}(U_1,\, U_\ell) - H \mathbf{B}(U_\ell,\, U_\ell)
				\\ &+\partial_\tau U_1 - \mathbf{P}_\alpha U_1 + H  \mathbf{B}(U_1,\, U_1) 
				\\ =& \partial_\tau U_1 + \mathbf{P}_\alpha U_\ell - H \widetilde{\mathbf{B}}(U_1,\, U_\ell) - H \mathbf{B}(U_\ell,\, U_\ell)
				\\ =& \partial_\tau U_1 +\partial_\tau U_\ell
				\\ =& \partial_\tau U_2.
			\end{aligned}
		\end{equation}
%
		
		Reversing the coordinate transformation (\ref{tra-1}) (\ref{tra-2}) yields two distinct global solutions $v_1,v_2$ to Equation (\ref{Eq-v}). By It\^{o}'s formula we get two solutions $u_1,u_2$ to (\ref{FNSE_0}) with same stochastic force $f$.
		Note that $ F \in C(( -\infty,\infty);L^2) $ by $\bar{U} \in C^\infty_c (\mathbb{R}^d),\ U_\ell \in C(( -\infty,\infty); H^{1} ) $ and $ \displaystyle\lim_{\tau \to -\infty} \|F(\tau)\|_L^2 <\infty$ by $\displaystyle\lim_{\tau \to -\infty} e^{\frac{2+d-2\alpha}{2\alpha} \tau } \|{U_\ell}(\tau)\|_{H^1}=0 $, then we have $f\in L^1( 0,T; L^2(\mathbb{R}^d)) \ \ \mathbb{P} -{\rm a.s.}$ for all $T <\infty$.
		It is not hard to check $u_i \in C_t L^2_x \cap L^2_t H^{\frac{\alpha}{2} }_x ((0,\infty) \times \mathbb{R}^d ) \  \mathbb{P}$-a.s. and $u_i$ satisfy energy equality (this is very similar as in the proof of Theorem \ref{Thm-1}).
	\end{proof}
	\begin{remark}
		The above procedure can be transfer smoothly onto 2D Euler equation, yielding a simpler proof for non-uniqueness than in \cite{Brue-Euler}. Seeing Appendix 4.2.
	\end{remark}
	
	\section{Appendix}
	\subsection{Proof of Proposition \ref{prop2.2}}\label{appendix-1}
	\begin{proof}[\textbf{proof}]
		We will divide the proof into three steps.
		
		\noindent{\bf Step 1:}
		Let conditions (2) and (3) in Proposition \ref{prop2.2} be satisfied, then $R(z, \mathbf{S} + \beta \mathbf{P})\eta \longrightarrow R(z, \mathbf{S} )\eta $ for all $\eta \in H$ and $z \in \mathbb{C}$ with $\operatorname{Re}(z) > b$.
		
		Consider the equation:
		\begin{equation}\label{L-E-2-1}
			\left\{
			\begin{aligned}
				&\partial_t U_\beta = \mathbf{S}U_\beta + {\beta} \mathbf{P} U_\beta,
				\\ &U_\beta|_{t=0} = \eta,
			\end{aligned}
			\right.
		\end{equation}
		where $\eta \in V$. 
		Let $U_\beta \in L_t^\infty L_x^2$ be the unique solution of equation (\ref{L-E-2-1})
		
		By energy estimate:
		\begin{equation}
			\partial_t \frac{1}{2} \|U_\beta\|^2 =  \langle {\beta}\mathbf{P}+\mathbf{S} U_\beta,U_\beta \rangle 
			\leq b \|U_\beta \|^2.
		\end{equation}
		Thus, $\|U_\beta (t) \|_{H} \leq e^{b t}\|U_0\|_{H}$. Therefore, the integral $
		\int_{0}^{\infty} e^{-t z} U_\beta (t, \cdot) \, {\rm d}t$ is well-defined when $\operatorname{Re} \, z > b$. Consequently, $R(z, \beta \mathbf{P} + \mathbf{S})U_0 = \int_{0}^{\infty} e^{-t z} U_\beta (t, \cdot) \, {\rm d}t$ exists.
		
		Next, we estimate the difference between the solutions of the two evolution equations. Let $U_0$ be the solution of equation (\ref{L-E-2-1}) when $\beta = 0$, we get that:
		\begin{equation}
			\begin{aligned}
				\partial_t \frac{1}{2} \|U_\beta -U_0 \|^2 
				&= \langle (\beta \mathbf{P} + \mathbf{S}) U_\beta -\mathbf{S} U_0 , U_\beta - U_0 \rangle 
				\\ &= \langle (\beta \mathbf{P} + \mathbf{S}) (U_\beta - U_0) , U_\beta - U_0 \rangle + \langle \beta \mathbf{P} U_0 , U_\beta - U_0 \rangle
				\\ &\leq b \|U_\beta - U_0\|^2
				+ \|\beta \mathbf{P} U_0 \| \ \| U_\beta - U_0 \|.
			\end{aligned}
		\end{equation}
		Then by Gronwall's inequality and condition (3) we have:
		\begin{equation}
			\begin{aligned}
				\|U_\beta(t)  -U_0(t)  \|
				\leq  {\beta}  e^{ b t} \int_{0}^{t} \| \mathbf{P}U_0(s)\| {\rm d}s 
				\leq {\beta}  e^{ b t} t \sup_{s\leq t}\| \mathbf{P}U_0(s)\|.
			\end{aligned}
		\end{equation}
		Thus, as $\beta \rightarrow 0$, we have $\displaystyle\sup_{t < T} \|U_\beta(t) - U_0(t)\| \to 0$ uniformly with respect to $T$.
		
		Note that
		\begin{equation}
			\begin{aligned}
				&\| \big(R(z,\beta \mathbf{P} +\mathbf{S}) - R(z, \mathbf{S}) \big)\eta_0 \| 
				\\ =& \left\| \int_{0}^{\infty} e^{-tz}  [U_\beta (t) -  U_0 (t)] {\rm d}t \right\| 
				\\ \leq& \int_{0}^{\infty} e^{-tz}  \| [U_\beta (t) -  U_0 (t)] \| {\rm d}t 
				\\ \leq& \int_{0}^{T}e^{-tz}  \| [U_\beta (t,\cdot) -  U_0 (t,\cdot)] \| {\rm d}t  + \int_{T}^{\infty}e^{-tz}  \| [U_\beta (t,\cdot) -  U_0 (t,\cdot)] \| {\rm d}t
				\\ \leq& \int_{0}^{T} e^{-z t+ b t } \beta \|\mathbf{P} U_0(t)\| {\rm d}t 
				+ \int_{T}^{\infty} e^{-z t} 2e^{b t} \|\eta\| {\rm d}t
				\\ \leq& \beta T \sup_{t\leq T}\|\mathbf{P} U_0(t)\| + Ce^{(b-z) T } \|U_0\|.
			\end{aligned}
		\end{equation}
		Let $T \to \infty$ first and $\beta \to 0$, we get that $R(z, \mathbf{S} + \beta \mathbf{P})\eta \longrightarrow R(z, \mathbf{S})\eta $ for any $\eta \in V$ with compact support.
		Note that $\| U_\beta (t)\| \leq e^{ bt }  \|\eta\|$, so $\displaystyle\sup_{\beta \in [0,1]} \|R(z, \beta \mathbf{P} + \mathbf{S})\| <\infty$ has a uniform bound. 
		Thus, for any $\eta \in H$, there exists $\eta_n \in V$ with $\eta_n \to \eta$ in $H$, we have $R(z, \beta \mathbf{P} + \mathbf{S})\eta_n \to R(z, \beta \mathbf{P} + \mathbf{S})\eta$, and the convergence is uniform with respect to $\beta$. Therefore, we also have $R(z, \mathbf{S} + \beta \mathbf{P})\eta \longrightarrow R(z, \mathbf{S})\eta$
		for all $\eta \in H$.

		\noindent {\bf Step 2:}
		Let all the conditions in Proposition \ref{prop2.2} be satisfied and $z \in \rho(\mathbf{L}_0): Re(z) > b $, then
		$R(z, \mathbf{L}_\beta)$ strongly converges to $R(z, \mathbf{L})$.
		
		Note that $ R(z, \mathbf{L}_\beta) = (1 - R(z, \mathbf{S} + \beta \mathbf{P})\mathbf{K})^{-1} R(z, \mathbf{S} + \beta \mathbf{P}) $. 
		By the Step 1 we know $R(z, \beta \mathbf{P} + \mathbf{S})\eta_0 \to R(z, \mathbf{S})\eta_0$ for any $\eta_0 \in H$. Now we consider $(1 - R(z, \mathbf{S} + \beta \mathbf{P})\mathbf{K})^{-1}$.
		
		We first show that $R(z, \beta \mathbf{P} + \mathbf{S})\mathbf{K} \to R(z, \mathbf{S})\mathbf{K}$ in the operator norm. In fact, the compact operator $\mathbf{K}$ can be approximated by a sequence of finite-dimensional operators $\mathbf{K}_n$. Since the strong convergence of finite-dimensional operators is equivalent to convergence in operator norm, $R(z, \beta \mathbf{P} + \mathbf{S})\mathbf{K}_n \to R(z, \mathbf{S})\mathbf{K}_n$ in the operator norm as well. Furthermore, since $\displaystyle\sup_{\beta \in \mathbb{R}} \|R(z, \beta \mathbf{P} + \mathbf{S})\| <\infty$, we get $R(z, \beta \mathbf{P} + \mathbf{S})\mathbf{K} \to R(z, \mathbf{S})\mathbf{K}$ in the operator norm.
		
		Moreover, note that the kernel of $1 - R(z, \mathbf{S}) \mathbf{K}$ is trivial for any $z \notin \operatorname{spec}\{\mathbf{L}\}$. Since $1 - R(z, \mathbf{S}) \mathbf{K}$ is a compact perturbation of the identity operator, it is a Fredholm operator, and thus $1 - R(z, \mathbf{S}) \mathbf{K}$ is invertible.
		It follows that $1 - R(z, \beta \mathbf{P} + \mathbf{S}) \mathbf{K}$ is also invertible when $\beta$ is big enough since $\|R(z,\beta \mathbf{P} + \mathbf{S} )\mathbf{K} -R(z,\mathbf{S} ) \mathbf{K}\| \to 0$. So $1-R(z, \beta \mathbf{P} + \mathbf{S})\mathbf{K} \to 1- R(z, \mathbf{S})\mathbf{K}$ in operator norm.
		
		Combining the convergence of the two parts, we conclude that $R(z, \mathbf{L}_\beta)$ strongly converges to $R(z, \mathbf{L})$.
		
		\noindent {\bf Step 3:}
		
		Now, let us prove Proposition \ref{prop2.2}.
		
		Let $\Gamma \subset \mathbb{R}^+$ be a simple closed curve that encloses $z_0$, $\Gamma \cap \sigma(\mathbf{L}) = \emptyset$ and $\mathbf{L}_0 \eta_0 = z_0 \eta_0$.
		By Step 1 and Step 2 we know $R(\lambda,\mathbf{L}_\beta) \to R(\lambda,\mathbf{L}) $ for any $\lambda \in \Gamma$. Moreover, since $\Gamma$ is a compact set, we have
		\begin{equation}
			\frac{1}{2\pi \text{i}} \int_\Gamma R(\lambda,\mathbf{L}_\beta) \eta_0 d\lambda
			\to \frac{1}{2\pi \text{i}} \int_\Gamma R(\lambda,\mathbf{L}_0) \eta_0 d\lambda  = \eta_0.
		\end{equation}
		So there exists $\beta_0 > 0$ such that $\frac{1}{2\pi \text{i}} \int_\Gamma R(\lambda,\mathbf{L}_{\beta_0}) d\lambda \eta_0 \neq 0$, which means there exists $z_{\beta_0} \in \sigma ( \mathbf{L}_{\beta_0} )$ and $ z_{\beta_0}$ is inside the curve $\gamma$.
		Now $ w_{ess}(\mathbf{L}_{\beta_0}) < 0 < \text{Re}(z_{\beta_0}) $. By Lemma \ref{lemma2.3} we get Proposition \ref{prop2.2} immediately.
	\end{proof}
	\subsection{2D Euler/fractional NS equations}
	We consider 2D Euler/fractional NS equations here:
	\begin{equation}\label{2D}
		\left\{
		\begin{aligned}
			&{\rm d} u + \mathbf{P}(  u\cdot \nabla u) {\rm d}t = \nu \Lambda^\alpha u {\rm d}t + f{\rm d}t,
			\\ &\nabla \cdot u = 0,
			\\ &u|_{t=0} = 0.
		\end{aligned}
		\right.	
	\end{equation}
	We have the following non-uniqueness theorem:
	\begin{theorem}\label{Thm-2D}
		Let $\alpha \in [0,1]$ and $\nu \geq 0$, there exist $f\in L^1_tL^2_x$ such that the equation (\ref{2D}) has two different weak solutions $u_1$ and $u_2$.
	\end{theorem}
	We only provide an outline of the proof, since the detail is very similar to the proof of Theorem \ref{Thm-1.1}.
	\begin{proof}[\textbf{proof}]
		First we use coordinate transformation (\ref{tra-1})(\ref{tra-2}) as well, then (\ref{2D}) change to:
		\begin{equation}\label{2D-2}
			\left\{
			\begin{aligned}
				&\partial_\tau {U} -\frac{\alpha-1}{\alpha} {U} - \frac{1}{\alpha} {\xi} \cdot \nabla_\xi {U} +  \mathbf{P}({U} \cdot \nabla_\xi {U}) = \nu \Lambda^\alpha_\xi {U} +{F} ,
				\\ &\nabla_\xi \cdot {U} =0,
				\\ &\lim_{\tau \to -\infty} e^{\frac{4-2\alpha}{2\alpha} \tau } \|{U}(\tau)\|_{L^2}=0.
			\end{aligned}
			\right.		
		\end{equation}
		By Proposition \ref{prop-3.1-a} \ref{prop2.2} we know there exist $\bar{U}$ such that $\mathbf{L}_{\bar{U}}$ has eigenvalue $z:{\rm Re} (z) =: a>0$ and characteristic function $\eta(\xi)$.
		So $U_l := {\rm Re} (e^{z \tau} \eta(\xi))$ is a solution for the following linear equation:
		\begin{equation}
			\left\{
			\begin{aligned}
				&\partial_\tau {U_l} = \mathbf{L}_{\bar{U}} U_l ,
				\\ &\nabla_\xi \cdot {U_l} =0,
				\\ &\lim_{\tau \to -\infty} \|{U_l}(\tau)\|_{L^2}=0.
			\end{aligned}
			\right.		
		\end{equation}
		Then we define:
		\begin{equation}
			\begin{aligned}
				U_1 :=& \bar{U} - \frac{1}{2} U_l,
				\\ U_2 :=& \bar{U} + \frac{1}{2} U_l,
				\\ F :=& - \mathbf{P}_\alpha \bar{U} + \mathbf{B}(\bar{U},\, \bar{U}) + \frac{1}{4} \mathbf{B}(U_l,\, U_l).
			\end{aligned}
		\end{equation}
		It's easy to check that $U_1,U_2$ are two different solution for equation (\ref{2D-2}), which is similar to (\ref{check}). Thus $u_1,u_2$, which are obtained through inverse coordinate transformation, are two different solutions for the equation (\ref{2D}).
	\end{proof}
	
\noindent{\bf  Acknowledgements}
This work was partially supported by 
National Key R$\&$D Program of China (No.2020YFA0712700);
Key Laboratory of Random Complex Structures and Data Science, Academy of  
Mathematics and Systems Science, Chinese Academy of  Sciences;
National Natural Science Foundation of China (Grant No.12090010,12090014,12471138).
%

\bibliographystyle{alpha}
\bibliography{ref}
\end{document}